\documentclass{amsart}

\usepackage{ae}
\usepackage[all]{xy}
\usepackage{graphicx,amsfonts,amssymb,amsmath}
\usepackage{eurofont}
\usepackage{natbib}

\usepackage[latin1]{inputenc}
\usepackage[T1]{fontenc}
\usepackage[english,francais]{babel}
\usepackage[cyr]{aeguill}

\usepackage[OT2,T1]{fontenc}
\DeclareSymbolFont{cyrletters}{OT2}{wncyr}{m}{n}
\DeclareMathSymbol{\sha}{\mathalpha}{cyrletters}{"58}
\input cyracc.def

 \newtheorem{thm}{Théorème}[section]
 \newtheorem*{thmA}{Théorème A}
 \newtheorem*{thmB}{Théorème B}
 \newtheorem*{thmC}{Théorème C}
 
 \newtheorem*{corD}{Corollaire D}

 \newtheorem{cor}[thm]{Corollaire}
 \newtheorem{lem}[thm]{Lemme}
 \newtheorem{prop}[thm]{Proposition}
 \theoremstyle{definition}
 
 \theoremstyle{remark}
 
 \theoremstyle{remark}
 
 \newtheorem{rem}[thm]{Remarque}
 \numberwithin{equation}{subsection}

 \newcommand{\To}{\longrightarrow}

 \renewcommand{\P}{\mathbb{P}}
 \newcommand{\Q}{\mathbb{Q}}
 \newcommand{\Z}{\mathbb{Z}}

\usepackage[dvipdfm]{hyperref}

\begin{document}

\title[Principe local-global pour les zéro-cycles]
 {Principe local-global pour les zéro-cycles sur certaines fibrations au-dessus de l'espace projectif}

\author{ Yongqi LIANG  }

\address{Yongqi LIANG \newline
Département de Mathématiques,  Bâtiment 425, Université  Paris-sud 11,  F-91405 Orsay, France}

\email{yongqi.liang@math.u-psud.fr}

\thanks{\textit{Mots clés} : zéro-cycle de degré $1$, principe de Hasse, approximation faible,
obstruction de Brauer-Manin, fibré en variétés de Severi-Brauer}

\thanks{\textit{Classification AMS} : 14G25 (11G35, 14D10)}

\date{\today}



\maketitle

\small
\textsc{Résumé.}
On étudie le principe local-global pour les zéro-cycles de degré $1$
sur certaines variétés définies sur les corps de nombres et fibrées au-dessus de l'espace projectif.

Parmi d'autres applications, on complète la preuve de l'assertion: l'obstruction de
Brauer-Manin est la seule au principe de Hasse et à
l'approximation faible pour les zéro-cycles de degré $1$ sur les
fibrés au-dessus de l'espace projectif en variétés de Severi-Brauer ou en surfaces de Châtelet.
\bigskip

\scriptsize
\begin{center}\textbf{LOCAL-GLOBAL PRINCIPLE FOR ZERO-CYCLES ON CERTAIN FIBRATIONS OVER THE PROJECTIVE SPACE}\end{center}
\small

\textsc{Abstract.} We study the local-global principle for zero-cycles of degree $1$ on certain varieties
defined over number fields and fibered over the projective space.

Among other applications, we complete the proof of the statement: the Brauer-Manin obstruction
is the only obstruction to the Hasse principle and
weak approximation
for zero-cycles of degree $1$
on Severi-Brauer variety bundles or Châtelet-surface bundles over the projective space.
\normalsize

\tableofcontents

\section*{Introduction}
Soit $X$ une variété projective lisse géométriquement intègre sur un corps de nombres
$k.$ Le principe de Hasse et l'approximation faible pour les points rationnels sur une telle
variété ont été considérés depuis longtemps.
L'obstruction de Brauer-Manin au principe de Hasse (resp. à l'approximation faible)
pour les points rationnels a été introduite par Yu. I. Manin dans son exposé \cite{Manin}
(resp. par Colliot-Thélène/Sansuc dans \cite{CTSansuc77-3}).
Parallèlement, pour les zéro-cycles, l'obstruction de Brauer-Manin est également définie
dans l'article de Colliot-Thélène \cite{CT95}.
On se demande si l'obstruction de Brauer-Manin est la seule obstruction au principe
de Hasse (resp. à l'approximation faible) pour les zéro-cycles de degré $1,$
voir \cite{CT99HP0-cyc} pour quelques conjectures explicites par Colliot-Thélène.
Mentionnons deux aspects des résultats directement liés à ce travail.

\begin{itemize}
\item[-]
Des résultats sur l'obstruction de Brauer-Manin pour les zéro-cycles ont été obtenus par
Colliot-Thélène, Eriksson, Saito, et Scharaschkin, dans
\cite{Saito}, \cite{CT99HP0-cyc}, \cite{Eriksson}, pour une courbe;
et par Colliot-Thélène, Frossard, Salberger, Skorobogatov, Swinnerton-Dyer, van Hamel, Wittenberg,
et l'auteur, dans \cite{Salberger}, \cite{CT-SD}, \cite{CT-Sk-SD}, \cite{CT99}, \cite{Frossard}, \cite{vanHamel},
\cite{Wittenberg}, \cite{Liang1}, pour certaines fibrations au-dessus d'une courbe,
voir l'introduction de \cite{Liang1} pour plus d'informations.

\item[-]
D'autre part, autour du problème parallèle de l'obstruction de Brauer-Manin pour les points rationnels
sur une fibration à fibres géométriquement intègres au-dessus de $\mathbb{P}^n,$
les meilleurs résultats généraux sont dus à Harari dans sa série d'articles
\cite{Harari}, \cite{Harari2}, et \cite{Harari3}.
Il impose une hypothèse arithmétique moins forte sur les fibres, à savoir que
l'obstruction de Brauer-Manin est la seule sur les fibres d'un sous-ensemble hilbertien.
Dans \cite{WittenbergLNM}, la même question pour les fibrations en variétés
de Severi-Brauer au-dessus de $\mathbb{P}^n,$ où des fibres géométriquement non intègres sont
permises, est discutée par Wittenberg en admettant l'hypothèse
de Schinzel.
\end{itemize}

Le but de ce travail est d'établir l'assertion que \og l'obstruction de Brauer-Manin est la seule au principe de
Hasse/à l'approximation faible pour les zéro-cycles de degré $1$\fg\mbox{ }pour certaines fibrations
au-dessus de $\mathbb{P}^n.$ Les résultats principaux sont
les suivants, où $X_{\bar{\eta}}=X_\eta\times_{k(\mathbb{P}^n)}\overline{k(\mathbb{P}^n)}$
est la fibre générique géométrique.

\begin{thmA}[Théorèmes \ref{thm10}, \ref{Pn-codim2}]
Soit $k$ un corps de nombres.
Soit $X\to\mathbb{P}^n$ un $k$-morphisme dominant à fibre générique géométriquement intègre
tel que $Br(X_{\bar{\eta}})$ soit fini et $Pic(X_{\bar{\eta}})$ soit sans torsion.

Supposons que

- toutes les fibres sont géométriquement intègres,

- pour tout point fermé $\theta$ dans un certain ouvert non vide de $\mathbb{P}^n,$
l'obstruction de Brauer-Manin est la seule au principe de Hasse/à l'approximation faible
pour les points rationnels (ou pour les zéro-cycles de degré $1$)
sur la fibre $X_\theta.$

Alors, l'obstruction de Brauer-Manin est la seule au principe de Hasse/à l'approximation faible
pour les zéro-cycles de degré $1$ sur $X.$
\end{thmA}

\begin{thmB}[Théorèmes \ref{thm15}, \ref{Pn-section}]
Soit $k$ un corps de nombres.
Soit $X\to\mathbb{P}^n$ un $k$-morphisme dominant à fibre générique géométriquement intègre
tel que $Br(X_{\bar{\eta}})$ soit fini et $Pic(X_{\bar{\eta}})$ soit sans torsion.

Supposons que

- la fibre générique ${X_\eta}_{/k(\mathbb{P}^n)}$ admet un zéro-cycle de degré $1,$

- pour tout point fermé $\theta$ dans un certain ouvert non vide de $\mathbb{P}^n,$
l'obstruction de Brauer-Manin est la seule à l'approximation faible
pour les zéro-cycles de degré $1$ sur la fibre $X_\theta.$

Alors, l'obstruction de Brauer-Manin est la seule à l'approximation faible
pour les zéro-cycles de degré $1$ sur $X.$
\end{thmB}

\begin{thmC}[Théorème \ref{Pn-codim1}]
Soit $k$ un corps de nombres.
Soit $X\to\mathbb{P}^n$ un $k$-morphisme dominant à fibre générique géométriquement intègre
tel que $Br(X_{\bar{\eta}})$ soit fini et $Pic(X_{\bar{\eta}})$ soit sans torsion.

Supposons que les conditions suivantes soient satisfaites:

- \emph{\textsc{(Abélienne-Scindée)}} pour tout point $\theta\in\mathbb{P}^n$ de codimension $1,$
il existe une composante irréductible $Y$ de la fibre $X_\theta$
de multiplicité $1$ telle que la fermeture algébrique de $k(\theta)$ dans
le corps de fonctions de $Y$ est une extension abélienne de $k(\theta),$

- pour tout point fermé $\theta$ dans un certain ouvert non vide de $\mathbb{P}^n,$
la fibre $X_\theta$ satisfait le principe de Hasse/l'approximation faible
pour les points rationnels (ou pour les zéro-cycles de degré $1$).

Alors, l'obstruction de Brauer-Manin est la seule au principe de Hasse/à l'approximation faible
pour les zéro-cycles de degré $1$ sur $X.$
\end{thmC}

L'hypothèse ``$Br(X_{\bar{\eta}})$ est fini et $Pic(X_{\bar{\eta}})$ est sans torsion''
est équivalente à l'hypothèse ``$H^i(X_{\bar{\eta}},\mathcal{O}_{X_{\bar{\eta}}})=0$
pour $i=1,2$ et $NS(X_{\bar{\eta}})$ est sans torsion'', qui est vérifiée
si $X_\eta$ est supposée rationnellement connexe (Lemme \ref{Pic-tors-Br-fini}).

\begin{corD}[Corollaires \ref{application-fibre-en-SB}, \ref{fibre en surfaces de chatelet}]
L'obstruction de Brauer-Manin est la seule au principe de Hasse/à l'approximation forte
pour les zéro-cycles de degré $1$ sur les fibrés au-dessus de l'espace projectif en variétés
de Severi-Brauer ou en surfaces de Châtelet.
\end{corD}

Le théorème A (resp. B) est la version pour les zéro-cycles des résultats de Harari:
\cite[Théorème 4.2.1]{Harari} et \cite[Théorème 1]{Harari3} (resp. \cite[Théorème 4.3.1]{Harari}).
Le théorème C est la version pour les zéro-cycles du théorème 3.5 de Wittenberg \cite{WittenbergLNM}.
Il a remarqué la validité du théorème $C$ sans en donner une preuve détaillée, ce travail confirme sa remarque.
On renvoie les lecteurs au texte ci-dessous pour les assertions plus précises.
Cet article comporte une discussion détaillée
autour l'arithmétique des zéro-cycles sur une fibration au-dessus de $\mathbb{P}^n.$
La méthode pour les questions sur les points rationnels ne s'étend pas directement pour résoudre
les questions sur les zéro-cycles, même si l'idée principale est similaire.  
Le théorème A (avec $n=1$) est aussi le point de départ des
résultats qui relient l'arithmétique des points rationnels et l'arithmétiques des zéro-cycles
sur les variétés rationnellement connexes, ceci sera expliqué dans l'article de l'auteur
\cite{Liang2short}.

Après quelques rappels dans \S \ref{notations}, on énonce et démontre les théorèmes
A et B pour le cas $n=1$ dans \S \ref{P1}; ensuite on énonce les théorèmes ci-dessus
sous forme plus détaillée et on les démontre dans \S \ref{Pn}; enfin on discute quelques applications dans \S \ref{appl}.


\section{Notations et rappels}\label{notations}

Dans tout ce travail, $k$ est toujours un corps de nombres. On note $\Omega_k$ l'ensemble
des places de $k.$ Pour chaque place $v\in\Omega_k,$
on note $k_v$ le corps local associé.
L'expression \og presque tout\fg\mbox{} signifie toujours
\og tout à l'exception d'un nombre fini\fg.

Soit $X$ une variété projective lisse et géométriquement intègre sur un corps $k,$ le composé de la
corestriction et de l'application d'évaluation définit un accouplement
\begin{equation*}
    \begin{array}{rcccl}
     \langle\cdot,\cdot\rangle_k:   Z_0(X) & \times & Br(X) & \to & Br(k),\\
          (\mbox{ }\sum_Pn_PP &,& b\mbox{ }) & \mapsto & \sum_Pn_Pcores_{k(P)/k}(b(P)),\\
    \end{array}
\end{equation*}
qui se factorise à travers l'équivalence rationnelle, où $Br(\cdot)=H^2_{\mbox{\scriptsize\'et}}(\cdot,\mathbb{G}_m)$
est le groupe de Brauer cohomologique.
Lorsque $k$ est un corps de nombres, on définit
l'\emph{accouplement de Brauer-Manin} pour les zéro-cycles:
\begin{equation*}
    \begin{array}{rcccl}
     \langle\cdot,\cdot \rangle _k:\prod_{v\in\Omega_k}Z_0(X_v) & \times & Br(X) & \to & \mathbb{Q}/\mathbb{Z},\\
         (\mbox{ }\{z_v\}_{v\in\Omega_k} & , & b\mbox{ }) & \mapsto & \sum_{v\in\Omega_k}inv_v(\langle z_v,b \rangle _{k_v}),\\
    \end{array}
\end{equation*}
où $inv_v:Br(k_v)\hookrightarrow\Q/\Z$ est l'invariant local en $v.$
On peut définir, pour les zéro-cycles
de degré $1$, l'obstruction de Brauer-Manin
au principe de Hasse, ou à l'approximation faible en un certain sens, \textit{cf.}  \S \ref{AF} ci-dessous et \cite{CT99HP0-cyc}
pour plus d'informations.

Soit $X$ une variété intègre sur un corps $k,$ un sous-ensemble $\textsf{Hil}\subset X$
de points fermés est un \emph{sous-ensemble hilbertien généralisé} s'il existe un morphisme
étale fini $Z\buildrel\rho\over\To U\subset X$ avec $U$ un ouvert non vide de $X$ et $Z$ intègre tel que
$\textsf{Hil}$ soit l'ensemble des points fermés $\theta$ de $U$ pour lesquels $\rho^{-1}(\theta)$ est connexe.
Soit $\textsf{Hil}_i$ ($i=1,2$) un sous-ensemble hilbertien généralisé, on
peut trouver un sous-ensemble hilbertien généralisé $\textsf{Hil}\subset\textsf{Hil}_1\cap \textsf{Hil}_2,$
\textit{cf.} \cite[\S 1.2]{Liang1}. On remarque un
sous-ensemble hilbertien généralisé $\textsf{Hil}$ est toujours non vide si $k$ est un corps
de nombres. En fait, en restreignant $U,$
on trouve un morphisme étale fini $U\to V$ où $V$ est un ouvert non vide de $\mathbb{P}^d$ avec
$d=dim(X).$ Son composé avec $Z\to U$ définit un sous-ensemble hilbertien généralisé $\textsf{Hil}'$ de $\mathbb{P}^d.$
Le théorème d'irréductibilité de Hilbert dit que $\textsf{Hil}'\cap\mathbb{P}^d(k)\neq\emptyset,$
qui implique immédiatement que $\textsf{Hil}\neq\emptyset.$

On fixe $\bar{k}_v$ une clôture algébrique de $k_v.$
Étant donné $P$ un point fermé de $X_v=X\times_kk_v,$ on fixe un $k_v$-plongement $k_v(P)\To \bar{k}_v,$
le point $P$ est vu comme un point $k_v(P)$-rationnel de $X_v.$
On dit qu'un point fermé $Q$ de $X_v$ est \emph{suffisamment proche}  de $P$ (par rapport à un voisinage $U_P$ de $P$ dans
l'espace topologique $X_v(k_v(P))$), si $Q$ a corps résiduel $k_v(Q)=k_v(P)$ et si l'on peut choisir un
$k_v$-plongement $k_v(Q)\To \bar{k}_v$ tel que $Q,$ vu comme un $k_v(Q)$-point rationnel de $X_v,$ soit contenu dans $U_P.$
En étendant $\mathbb{Z}$-linéairement, cela a un sens de dire que $z'_v\in Z_0(X_v)$ est suffisamment proche de $z_v\in Z_0(X_v)$
(par rapport à un système de voisinages des points qui apparaissent dans le support de $z_v$).
D'après la continuité de l'accouplement de Brauer-Manin, \textit{cf.} \cite[Lemma 6.2]{B-Demarche},
pour un sous-ensemble \emph{fini} $B\subset Br(X_v),$ on a
$\langle z'_v,b \rangle _v= \langle z_v,b \rangle _v\in Br(k_v)$ pour tout $b\in B$ si
$z'_v$ suffisamment proche de $z_v$ (par rapport à $B$).

\subsection{Approximation pour les zéro-cycles}\label{AF}

Soit $X$ une variété projective lisse et géométriquement intègre de dimension $d$ sur un corps de nombres $k.$
Pour une place $v\in\Omega_k,$ on pose $X_v=X\times_kk_v.$
On considère, pour un entier $m$ positif non nul, la flèche
$$CH_0(X_v)\to CH_0(X_v)/m.$$
On dit que \emph{l'obstruction de Brauer-Manin est la seule à l'approximation faible
(resp. forte) (au niveau du groupe de Chow)}
pour les zéro-cycles de degré $\delta$ sur $X,$
si pour tout entier positif non nul $m,$ pour tout ensemble fini $S\subset\Omega$ (resp. pour $S=\Omega$),
et pour toute famille $\{z_v\}_{v\in\Omega_k}$ de zéro-cycles locaux de degré $\delta$ orthogonale au groupe de
Brauer $Br(X),$
il existe un zéro-cycle global $z$ sur $X$ de degré $\delta$ tel que $z$ et $z_v$ aient la même image dans
$CH_0(X_v)/m$ pour toute $v\in S.$
Dans ce travail on omet la phrase \og \emph{au niveau du groupe de Chow} \fg\mbox{ } quand on parle d'approximation faible
pour les zéro-cycles.

\subsection{Lemmes de déplacement}
On dit qu'un zéro-cycle est \emph{séparable} s'il est écrit comme une combination finie $\Z$-linéaire
de points fermés sans multiplicité.
On va appliquer plusieurs fois les lemmes de déplacement suivants.

\begin{lem}\label{deplacement}
Soient $X$ une variété  intègre régulière  sur un corps parfait infini $k,$ et $U$ un ouvert non vide de $X.$
Alors tout  zéro-cycle $z$ de $X$ est rationnellement équivalent, sur $X,$ à un zéro-cycle $z'$ à support dans $U.$
\end{lem}

\begin{proof}
On trouve une démonstration détaillée dans \cite{CT05} \S 3.
\end{proof}

\begin{lem}\label{deplacement2}
Soit $\pi:X\To\mathbb{P}^1$ un morphisme non constant au-dessus de la droite projective sur
$\mathbb{R},$ $\mathbb{C}$ ou sur un corps $p$-adique,
avec $X$ une variété lisse intègre.
Soient $D$ un ensemble fini de points fermés de $\mathbb{P}^1,$ $X_0$ un ouvert de Zariski non vide de $X.$

Alors, pour tout zéro-cycle $z$ à support dans $X_0,$ il existe un zéro-cycle séparable $z'$
à support dans $X_0$ tel que $z'$ soit suffisamment proche de $z$ et tel que $\pi_*(z')$ soit séparable à support en dehors de $D.$
Les zéro-cycles $\pi_*(z)$ et $\pi_*(z')$ sont rationnellement équivalents.
\end{lem}

\begin{proof}
Essentiellement, ce résultat se déduit du théorème des fonctions implicites. On trouve les arguments à la page 19 de
\cite{CT-Sk-SD} et 89 de \cite{CT-SD}.
Les zéro-cycles $\pi_*(z)$ et $\pi_*(z')$
sont automatiquement rationnellement équivalents sur $\mathbb{P}^1$ car ils ont le même degré.
\end{proof}

\subsection{Lemme d'irréductibilité de Hilbert pour les zéro-cycles}

Le lemme suivant est une version plus fine du lemme 3.4 de \cite{Liang1} appliqué à $\mathbb{P}^1.$
Il est en un certain sens une version effective du théorème
d'irréductibilité de Hilbert pour les zéro-cycles,
c'est aussi la version pour les zéro-cycles du théorème 1.3 de Ekedahl \cite{Ekedahl}.

\begin{lem}\label{simplified}
Soit $k$ un corps de nombres.
On fixe un $k$-point $\infty\in\mathbb{P}^1,$ et on note $\mathbb{A}^1=\mathbb{P}^1\setminus\{\infty\}.$
Soit $\textsf{Hil}$ un sous-ensemble hilbertien généralisé de $\mathbb{A}^1.$
Soient $S$ un sous-ensemble fini non vide de $\Omega_k$ et $v_0$ une place non-archimédienne hors de $S.$

Soit $z_v\in Z_0(\mathbb{P}^1_v)$ un zéro-cycle effectif séparable de degré $d>0$ supporté dans $\mathbb{A}^1$
pour toute $v\in S.$

Alors il existe un point fermé $\theta$ de $\mathbb{A}^1$ de degré $d,$ tel que

(1) $\theta\in\textsf{Hil},$

(2) $\theta$ soit entier en dehors de $S\cup\{v_0\},$

(3) $\theta$ soit suffisamment proche de $z_v$ pour toute $v\in S.$
\end{lem}

\begin{proof}
Pour chaque $v\in S,$ on peut écrire $z_v-d\infty=div(f_v)$ avec $f_v$ une fonction rationnelle de $\mathbb{P}^1_v,$
autrement-dit $f_v$ est un polynôme à coefficients dans $k_v$ unitaire, de degré $d.$
D'après l'approximation forte pour un corps de nombres, il existe un polynôme $f$ à coefficients dans $k$
unitaire de degré $d$ tel que

(i)$f$ soit suffisamment proche de $f_v$ pour tout $v\in S,$

(ii)$f$ soit à coefficients entiers en dehors de $S\cup\{v\}.$

On écrit $z'-d\infty=div(f),$ grâce au lemme de Krasner
le zéro-cycle effectif $z'$ est suffisamment proche de $z_v$ pour $v\in S,$
il est alors séparable.

Le polynôme $f$ définit un $k$-morphisme fini $F:\mathbb{P}^1\to\mathbb{P}^1$ de degré $d$
tel que $F^*(\infty)=d\infty$ et $F^*(0)=z'$
Supposons que le sous-ensemble hilbertien généralisé $\textsf{Hil}$ est défini par
un morphisme quasi-fini $Z\to\mathbb{A}^1$ avec $Z$ une variété intègre.
La composition $\varphi:Z\to\mathbb{A}^1\buildrel{F}\over\to\mathbb{A}^1$ définit un sous-ensemble hilbertien
généralisé $\textsf{Hil}'$ de $\mathbb{A}^1.$ En restreignant $\varphi$ à
un certain ouvert non vide si nécessaire, on peut supposer de plus que
$\theta=F^{-1}(\theta')\in \textsf{Hil}$ une fois qu'on a $\theta'\in \textsf{Hil}'.$
D'après le théorème d'irréductibilité de Hilbert (de version effective par Ekedahl \cite[Theorem 1.3]{Ekedahl}),
il existe un $\theta'\in \textsf{Hil}'\cap\mathbb{A}^1(k)$ suffisamment proche de $0\in\mathbb{A}^1(k_v)$ pour toute $v\in S$
et de plus $\theta'$ est entier en dehors de $S\cup\{v_0\}.$
On prend $\theta=F^{-1}(\theta')\in \textsf{Hil},$ il est suffisamment proche de $z'$ et alors suffisamment proche de $z_v$ pour toute
$v\in S.$ De plus le point $\theta\in\mathbb{A}^1$ est défini par le polynôme $f-\theta',$
il est alors entier en dehors de $S\cup\{v_0\}.$
\end{proof}

\subsection{Variétés rationnellement connexes}

On rappelle la notion de connexité rationnelle au sens de Koll\'ar, Miyaoka et Mori \cite[IV.3]{Kollar}.
Une variété $X$ définie sur un corps $k,$ quelconque de caractéristique nulle, est dite \emph{rationnellement connexe},
si pour toute paire des points $P,Q\in X(L)$ il existe un $L$-morphisme $f:\mathbb{P}^1_L\to X_L$ tel que
$f(0)=P$ et $f(\infty)=Q,$
où $L$ est un certain corps algébriquement clos non dénombrable et contenant $k.$

Le lemme suivant est connu depuis longtemps, on inclut une preuve ici pour le confort du lecteur.
Il nous permet de vérifier dans les théorèmes de cet article l'hypothèse que
$PicX_{\bar{\eta}}$ est sans torsion et $Br(X_{\bar{\eta}})$ est fini.

\begin{lem}\label{Pic-tors-Br-fini}
Soit $X$ une variété projective connexe et lisse sur un corps $k$ algébriquement clos de caractéristique nulle.
Si $X$ est rationnellement connexe, alors son groupe de Picard $Pic(X)$ est sans torsion, et
son groupe de Brauer $Br(X)$ est fini.
\end{lem}

\begin{proof}
La connexité rationnelle implique que $H^i(X,\mathcal{O}_X)=0(i=1,2)$ et $\pi_1(X)=0,$
 \textit{cf.} le corollaire 4.18 de \cite{Debarre}
et sa preuve, ceci permet de conclure.
On donne une preuve alternative.

Comme $X$ est rationnellement connexe, son groupe fondamental $\pi_1(X)$ est nul, \cite{Debarre},
Corollaire 4.18(b). Il existe une suite exacte
$$0\to(NS(X)_{tors})^*\to\pi_1^{ab}(X)/n\pi_1^{ab}(X)\to (Pic^o(X)_n)^*\to0$$
pour tout entier positif $n$ suffisamment divisible, où $-^*$ est le dual de Pontryagin $Hom(-,\mathbb{Q}/\mathbb{Z}),$
\textit{cf.} la preuve du corollaire III.4.19(b)
de \cite{MilneEC}. On trouve alors que le groupe de Néron-Severi $NS(X)$ est sans torsion, et que
$Pic^o(X)_n=0$ pour $n$ suffisamment divisible, le groupe de Picard $Pic(X)$ est donc sans torsion.

Montrons la finitude de $Br(X)$ avec une astuce du point générique proposée par Colliot-Thélène.
On note $\eta$ le point générique de $X$ et $K=k(\eta)$ le corps des fonctions de $X,$
on pose $X_K=X\times_kK.$ On évalue l'image $b'$ de $b\in Br(X)$ dans $Br(X_K)$ en le point
$\eta\in X_K(K),$ on obtient un élément de $Br(K)$ qui est exactement l'image de $b$
par l'inclusion naturelle $Br(X)\hookrightarrow Br(K).$ D'un autre côté, on évalue
$b'$ en un point $P\in X(k)\subset X(K),$ on obtient $0\in Br(K)$ car l'évaluation se factorise à travers $Br(k)=0.$
Comme $X$ est rationnellement connexe, la classe du zéro-cycle $\eta-P_K\in CH_0(X_K)$ est annulée par
un certain nombre entier $m\neq0$ d'après la proposition 11 de \cite{CT05}.
L'évaluation de $b'$ en cette classe donne $mb=0$ dans $Br(K),$ donc $Br(X)$ est annulé par $m.$
Comme $Br(X)$ est de type cofini, \textit{cf.} pages 80-81 de \cite{Br} , il est alors fini.
\end{proof}

Soient $k$ un corps de nombres et $X\to\mathbb{P}^1$ un $k$-morphisme dominant à fibre
générique rationnellement connexe. Pour une place $v$ de $k,$
comme l'application $CH_0(X_v)\to CH_0(\mathbb{P}^1_v)$ s'identifie à l'application de degré,
la proposition suivante est une conséquence immédiate du théorème 5 de Koll\'ar/Szab\'o \cite{Kollar-Szabo}.
Voir aussi le corollaire 2.2 de \cite{Wittenberg} pour une généralisation par Wittenberg.

\begin{prop}\label{H-CH0}
Soit $X\to\mathbb{P}^1$ une fibration au-dessus de la droite projective sur un corps de nombres $k.$
Si la fibre générique est rationnellement connexe, alors

(H CH0) l'application induite $CH_0(X_v)\to CH_0(\mathbb{P}^1_v)$ est un isomorphisme pour presque toute place $v$ de $k.$
\end{prop}

Grâce à cette proposition, on se ramène à l'approximation \emph{faible}
quand on considère l'approximation \emph{forte} pour les zéro-cycles sur
une fibration au-dessus de $\mathbb{P}^1$ à fibre générique rationnellement connexe.

\subsection{Lemme formel}

Pour le lemme formel suivant,
on trouve la version originale de Harari pour les points rationnels dans {\cite[Corollaire 2.6.1]{Harari}} et une version pour les
zéro-cycles dans {\cite[Lemma 4.5]{CT-Sk-SD}}.
\begin{lem}\label{lemme formel}
Soit $X$ une variété projective, lisse, et géométriquement intègre sur un corps de nombres $k.$
Soient $U$ un ouvert non vide de $X$ et $\{A_1,\ldots,A_n\}\subset Br(U)\subset Br(k(X)).$
Soit $r$ un nombre entier. On note $B$ l'intersection dans $Br(k(X))$ du sous-groupe $Br(X)$ et du sous-groupe engendré par
les $A_i.$

On suppose que pour chaque $v\in\Omega_k,$ il existe un zéro-cycle $z_v$ sur $X_v$ de degré $r$ supporté dans $U_v$ tel que
la famille $\{z_v\}_{v\in\Omega}$ est orthogonale à $B.$

Alors, pour tout ensemble fini $S$ de places de $k,$ il existe un ensemble fini $S'$ de places de $k$ contenant $S$
et pour chaque $v\in S'$ un zéro-cycle $z'_v$ sur $U_v$ de degré $r$ tels que
$$\sum_{v\in S'}inv_v(\langle A_i,z'_v\rangle_v)=0$$
et de plus $z'_v=z_v$ pour toute $v\in S.$
\end{lem}


\section{Fibrations au-dessus de $\mathbb{P}^1$}\label{P1}

Dans cette section, on considère des fibrations $X\to\mathbb{P}^1,$ où dans tout ce texte le mot \og fibration\fg\mbox{ }
signifie un morphisme dominant à fibre générique géométriquement intègre.
L'obstruction de Brauer-Manin pour les points rationnels est discutée dans une série d'articles
de Harari \cite{Harari}, \cite{Harari2}, \cite{Harari3}, au lieu de supposer que \og beaucoup de\fg
\mbox{} fibres satisfont le
principe de Hasse (resp. l'approximation faible) on suppose seulement que
l'obstruction de Brauer-Manin est la seule pour ces fibres.
Dans cette section, on développe une
version concernant les zéro-cycles de degré $1$ parallèle aux théorèmes de Harari.
Dans l'article \cite{Harari}, deux situations ont été considérées. Dans le théorème
4.2.1 de \cite{Harari}, on suppose que toutes les fibres sont géométriquement intègres
(sauf la fibre au-dessus de $\infty\in\mathbb{P}^1(k)$), le théorème parallèle \ref{thm10} est montré
dans \S \ref{subsection2}. Dans le théorème 4.3.1 de \cite{Harari}, on suppose qu'il
existe un $k(\mathbb{P}^1)$-point rationnel sur la fibre générique, pour les zéro-cycles on montre dans \S \ref{subsection1}
le théorème parallèle \ref{thm15}, où on suppose qu'il existe un zéro-cycle de degré $1$ sur
la fibre générique. D'abord, avant les théorèmes principaux de cette section, on donne des résultats sur
la spécialisation du groupe de Picard et du groupe de Brauer.


\subsection{Flèches de spécialisation}\label{specialisation}

On considère une fibration $X\to Y$ avec $X$ et $Y$ des variétés propres lisses et géométriquement intègres,
on note $K=k(Y)$ le corps des
fonctions, et $\eta=Spec(K)$ le point
générique de $Y,$ et en plus $X_{\bar{\eta}}=X_\eta\times_K\bar{K}.$
Il existe un ouvert dense $Y_0$ de $Y$ tel que pour tout point $\gamma\in Y_0$ la fibre $X_\gamma$ soit
lisse et géométriquement intègre.
Il existe des flèches de spécialisation $Pic(X_{\bar{\eta}})\to Pic(X_{\bar{\gamma}})$
et $Br(X_{\bar{\eta}})\to Br(X_{\bar{\gamma}}),$ où $X_{\bar{\gamma}}=X_\gamma\times_{k(\gamma)}\overline{k(\gamma)}.$
Si l'on suppose de plus que $Br(X_{\bar{\eta}})$ est fini et $Pic(X_{\bar{\eta}})$ est sans torsion,
quitte à restreindre $Y_0,$ ces deux flèches sont des isomorphismes pour tout $\gamma\in Y_0,$
\textit{cf.} \cite[Proposition 3.4.2]{Harari} et \cite[Proposition 2.1.1]{Harari2}. Alors
$Br(X_{\bar{\gamma}})$ est fini et $Pic(X_{\bar{\gamma}})$ est sans torsion pour tout $\gamma\in Y_0.$

Afin de comparer les groupes de Brauer \emph{sur le corps de base}, la notion d'un sous-ensemble hilbertien généralisé
intervient.
On considère une fibration $X\to Y=\mathbb{P}^1.$
Pour un élément $b$ de $Br(K(X_\eta))=Br(k(X)),$ il existe un ouvert non vide $X_0$ de $X$ tel que
$b\in Br(X_0).$ Pour presque tout point $\theta$ de $\mathbb{P}^1,$ la fibre
$X_\theta$ est géométriquement intègre et son point générique $\eta(X_\theta)$ est dans $X_0,$ la spécialisation
de $b$ au point $\eta(X_\theta)$ est alors un élément de $Br(k(\theta)(X_\theta)).$
Lorsque ${Br(X_\eta)}/{Br(K)}$ est un groupe fini,
ceci définit
la flèche de spécialisation:
$$sp_\theta:\frac{Br(X_\eta)}{Br(K)}\To \frac{Br(X_\theta)}{Br(k(\theta))},$$
pour presque tout point $\theta$ de $\mathbb{P}^1,$
\textit{cf.} \cite[\S 3.3]{Harari} pour plus d'informations.

\begin{prop}[Harari]\label{Br-specialisation}
Soit $X\To\mathbb{P}^1$ une fibration sur un corps de nombres $k.$
On suppose que $Br(X_{\bar{\eta}})$ est fini et que $Pic X_{\bar{\eta}}$ est sans torsion.

Alors, il existe un sous-ensemble hilbertien généralisé \textsf{Hil} de $\mathbb{P}^1$ tel que
pour tout  $\theta\in\textsf{Hil},$ la flèche de spécialisation
$$sp_\theta:\frac{Br(X_\eta)}{Br(K)}\To \frac{Br(X_\theta)}{Br(k(\theta))}$$
est un isomorphisme de groupes abéliens finis.
\end{prop}

\begin{proof}
Les preuves du théorème 3.5.1 de \cite{Harari} et du théorème 2.3.1 de \cite{Harari2} fonctionnent bien
pour un point fermé $\theta$ au lieu d'un point $k$-rationnel.
Dans le théorème 3.5.1 de \cite{Harari}, on suppose que $Br(X_{\bar{\eta}})$ est nul, tandis que dans l'article
\cite{Harari2}, le théorème concerné est renforcé  avec l'hypothèse que $Br(X_{\bar{\eta}})$ est fini.
\end{proof}

\begin{rem}
Si le groupe $Br(X_{\bar{\eta}})$ n'est pas fini, la cohomologie $H^2(X_{\bar{\eta}},\mathcal{O}_{X_{\bar{\eta}}})$
est non nulle, \textit{cf.} \cite[II. Cor.3.4]{Br}.
À cause de ceci, on ne peut pas comparer (via le résultat de Grothendieck sur les faisceaux inversibles)
les groupes de Picard de la fibre générique géométrique et des fibres spéciales géométriques.
Par conséquent, on ne peut pas contrôler la partie algébrique des groupes de Brauer,
même si la partie invariante par l'action galoisienne de $Br(X_{\bar{\eta}})$ est finie.

Par exemple, si la fibre générique $X_\eta$ est une surface K3, même si le groupe $Br(X_\eta)/Br(K)$ est fini
\cite[Theorem 1.2]{SkorobogatovZarhinK3}, on ne sait pas si la flèche
$sp_\theta$ est surjective pour un certain $\theta.$
\end{rem}

\subsection{Cas où les fibres sont géométriquement intègres}\label{subsection2}
On suppose que toutes les fibres de la fibration considérée $X\to\P^1$ contiennent une composante
irréductible de multiplicité un qui est géométriquement intègre.

On suit la stratégie de la preuve du théorème 2 de Harari \cite{Harari3}.
D'abord, on montre la proposition cruciale suivante
qui est parallèle à la proposition 1 de \cite{Harari3}. Ensuite, on déduit le théorème \ref{thm10}.

\begin{prop}\label{lem-complicated}
Soit $\pi:X\To\mathbb{P}^1$ une fibration.
On suppose que toute fibre contient une composante irréductible de multiplicité un qui est géométriquement intègre.
On suppose que $U$ est un ouvert non vide de $X$ et $\Lambda\subset Br(U)$ un sous-ensemble fini d'éléments
de $Br(k(X)).$ On fixe un $k$-point $\infty$ de $V=\pi(U)$ tel que la fibre $X_\infty$ soit lisse et géométriquement
intègre, on note  $U_0=U\setminus X_\infty$ et $V_0=V\setminus\{\infty\}.$

Alors, pour un entier positif $d$ fixé, il existe un ensemble fini $S\subset\Omega_k,$ un
ensemble infini $\Sigma\subset\Omega_k,$ et un ensemble fini
$E$ d'éléments de $Br(K)\cap Br(U_0)\subset Br(k(X))$ tels que
si $S'\supset S$ est un ensemble fini de places de $k,$ il existe un nombre fini de places
$v_1,\ldots,v_l,v_\infty$ hors de $S'$ avec $v_\infty\in\Sigma$ et
$\theta_i\in Z_0(\mathbb{P}^1_{v_i})(i\in\{1,\ldots,l\}\cup\{\infty\})$ zéro-cycles
effectifs de degré $d$ ayant la propriété suivante:

si $\theta$ est un point fermé de $V_0\subset\mathbb{A}^1=\mathbb{P}^1\setminus\{\infty\}$ de degré $d,$
suffisamment proche de chaque
$\theta_i(i\in\{1,\ldots,l\}\cup\{\infty\})$ comme zéro-cycles locaux, entier en dehors
de $S'\cup\{v_1,\ldots,v_l,v_\infty\}\cup\Sigma,$
et de plus si la fibre $X_\theta\cap U_0$ possède des $k(\theta)_w$-points lisses $M_w$ pour
toute $w\in S'\otimes_kk(\theta),$ vérifiant $\sum_{w\in S'\otimes_kk(\theta)}inv_w(A(M_w))=0$ pour
$A\in\Lambda\cup E,$ alors $X_\theta\cap U_0$ possède des $k(\theta)_w$-points lisses $M_w$
pour $w\in (\Omega\setminus S')\otimes_kk(\theta)$ vérifiant
$$\sum_{w\in \Omega_{k(\theta)}}inv_w(A(M_w))=0\mbox{ pour }A\in\Lambda\cup E.$$
\end{prop}

\begin{rem}
Dans cette proposition, les mêmes $S,$ $E,$ et $\Sigma$ fonctionnent pour tout entier positif $d,$ même si
l'on va fixer un entier positif $d$  avant l'application de cette proposition dans la preuve du théorème \ref{thm10}.
\end{rem}

\begin{proof}
On note $Z=X\setminus U=Z_1\cup Z_2$ où les composantes irréductibles du fermé $Z_1$ dominent $\mathbb{P}^1$
tandis que $Z_2$ est contenu dans la réunion d'un ensemble fini de fibres $X_{m_i}(1\leqslant i\leqslant l)$
et on peut supposer que $Z_2=\bigsqcup_{i=1}^lX_{m_i}.$ En restreignant $U$ si nécessaire, on peut supposer
que pour tout $\theta\in\mathbb{P}^1$ différent de chaque $m_i,$ la fibre $X_\theta$ est lisse et géométriquement
intègre.
On sait alors que le point $\infty\in\mathbb{P}^1$ est différent de chaque $m_i(i\in\{1,\ldots,l\}),$ on écrit
$\mathbb{A}^1=\mathbb{P}^1\setminus\{\infty\}=Spec(k[T]).$ Pour $i\in\{1,\ldots,l\},$
on note $P_i(T)$ le polynôme irréductible unitaire définissant
le point fermé $m_i\in\mathbb{A}^1,$ et $k_i=k(m_i)=k[T]/P_i(T)$ le corps résiduel de
$m_i.$  La fibre $X_{m_i}$  a une composante irréductible de multiplicité un qui est géométriquement intègre,
on fixe une telle composante $X^{irr}_{m_i}$  et on note $K_i$ son corps des fonctions.
Le corps $k_i$ est alors algébriquement fermé dans
$K_i.$
Pour $A\in Br(k(X)),$ on note $\partial_{A,i}$ son image par l'application de résidu
$Br(k(X))\To H^1(K_i,\mathbb{Q}/\mathbb{Z}).$ Les éléments $\partial_{A,i}(A\in\Lambda)$ engendrent un sous-groupe
fini abélien de la forme $G_i=H^1(Gal(K'_i/K_i),\mathbb{Q}/\mathbb{Z})$ où $K'_i$ est une extension finie abélienne
de $K_i.$ On note $k'_i$ la fermeture intégrale de $k_i$ dans $K'_i.$
On dispose du sous-groupe $G'_i=H^1(Gal(K_ik'_i/K_i),\mathbb{Q}/\mathbb{Z})\simeq H^1(Gal(k'_i/k_i),\mathbb{Q}/\mathbb{Z})$
de $G_i.$ On note également $K_\infty$ le corps des fonctions de la fibre $X_\infty$ et $k_\infty=k.$

\medskip
\emph{Étape 1: Construction de l'ensemble $E$ et $\Sigma.$}

On considère la suite exacte de Faddeev (\emph{cf.} \cite[Corollary 6.4.6]{GilleSzamuely})
$$0\To Br(k)\To Br(k(T))\buildrel{\partial_\theta}\over\To\bigoplus_{\theta\in\mathbb{A}^{1(1)}} H^1(k(\theta),\mathbb{Q}/\mathbb{Z})\To0$$
où $\mathbb{A}^{1(1)}$ est l'ensemble des points fermés de $\mathbb{A}^1.$
Pour $i\in\{1,\ldots,l\},$ on définit un sous-groupe
$$E'_i=\{\beta\in Br(k(T));\partial_\theta(\beta)=0\mbox{ si }\theta\neq m_i\mbox{ et }\partial_{m_i}(\beta)\in G'_i=H^1(Gal(k'_i/k_i),\mathbb{Q}/\mathbb{Z})\}.$$
Le groupe $E'_i/Br(k)$ est alors fini, on prend un ensemble fini $E_i\subset Br(k(T))$ de représentants,
et on note $E=\bigsqcup_{i=1}^lE_i.$
Son image dans $Br(k(X))$ est contenue dans $Br(U_0).$

D'après le théorème de Tsen, le groupe $Br(\bar{k}(T))$ est nul,
il existe donc une extension finie $k'$ de $k$ telle que les restrictions des éléments de $E$ dans $Br(k'(T))$
soient nulles. On définit $\Sigma$ comme l'ensemble des places non-archimédiennes de $k$ qui sont
totalement décomposées dans $k',$ c'est un ensemble infini d'après le théorème de \v{C}ebotarev.

\medskip
\emph{Étape 2: Extension aux modèles entiers.}

On peut trouver un modèle entier
$\mathcal{U}_0$ (resp. $\mathcal{U},$ $\mathcal{V}_0,$ $\mathcal{V},$ $\mathcal{X},$ $\mathcal{X}_{m_i},$ $\mathcal{X}_\infty,$ et
$\mathcal{Z}_1$) de $U_0$
(resp. $U,$ $V_0,$ $V,$ $X,$  $X_{m_i}^{irr},$ $X_\infty,$ et $Z_1$)
sur $\mathcal{W}_1=Spec(O_{k,S_1})$ avec
l'ensemble fini $S_1$ contenant toutes les places archimédiennes et toutes les places ramifiées dans les
extensions $k'_i/k.$
On peut trouver pour chaque $i\in\{1,\ldots,l\}$ un ouvert lisse $\mathcal{U}_i$ de
$\mathcal{X}_{m_i}$  disjoint
de $\mathcal{Z}_1,$ tel que les éléments de $H^1(Gal(K'_i/K_i),\mathbb{Q}/\mathbb{Z})$ soient dans
$H^1_{\mbox{\scriptsize\'et}}(\mathcal{U}_i,\mathbb{Q}/\mathbb{Z}).$
On trouve aussi $\mathcal{U}_\infty$ un ouvert non vide de $\mathcal{X}_{\infty}$ disjoint de $\mathcal{Z}_1,$
tel que les résidus $\partial_{A,\infty}$ des éléments $A\in\Lambda\cup E\subset Br(k(X))$
dans $H^1(K_\infty,\mathbb{Q}/\mathbb{Z})$ soient dans
$H^1_{\mbox{\scriptsize\'et}}(\mathcal{U}_\infty,\mathbb{Q}/\mathbb{Z}).$
On note $\widetilde{m_i}$ (resp. $\widetilde{\infty}$) l'adhérence schématique de $m_i$ (resp. $\infty$) dans
$\mathbb{P}^1_{\mathcal{W}_1},$ quitte à augmenter $S_1,$ ces fermés et l'ouvert
$\mathcal{V}_0$ forment une partition de $\mathbb{P}^1_{\mathcal{W}_1}.$
On peut supposer également, quitte à augmenter $S_1,$ que
\begin{itemize}
\item
$\mathcal{X},$ $\mathcal{X}_\infty,$ et $\mathcal{U}_i$ sont lisses sur $\mathcal{W}_1,$
\item
$\mathcal{X}$ et $\mathcal{Z}_1$ sont plats sur $\mathbb{P}^1_{\mathcal{W}_1},$
\item
$\mathcal{X}_{m_i}$ et toutes les composantes irréductibles de $\mathcal{X}_{m_i}\setminus\mathcal{U}_i$
(resp. $\mathcal{X}_\infty\setminus\mathcal{U}_\infty$)
sont plates sur $\widetilde{m_i}$ (resp. $\widetilde{\infty}$),
\item
toutes les fibres
de $\mathcal{U}_0\to\mathcal{V}_0,$ de $\mathcal{U}_i\to\widetilde{m_i},$ et de
$\mathcal{X}_\infty\to\widetilde{\infty}$ sont géométriquement intègres,
\item
les réductions modulo
$v\notin S_1$ des fibres $\mathcal{X}_{m_i}$ et $\mathcal{X}_\infty$
sont deux à deux disjointes, avec en plus $\mathcal{X}_{m_i}\cap \mathcal{U}_0=\emptyset$ et
$\mathcal{X}_\infty\cap \mathcal{U}_0=\emptyset,$
\item
les éléments de $\Lambda\cup E\subset Br(U_0)$ s'étendent en des élément de $Br(\mathcal{U}_0),$
\item
le polynôme $P_i(T)$ est à coefficients dans $O_{k,S_1},$ et que sa réduction modulo $v\notin S_1$ est un polynôme
séparable.
\end{itemize}

On peut aussi supposer qu'il existe un revêtement étale fini galoisien (connexe) $\mathcal{Y}_i$ de $\mathcal{U}_i$
de groupe $Gal(K'_i/K_i),$ qui se factorise par un revêtement étale fini galoisien $\mathcal{Y}_i\To\mathcal{Y}'_i$
de groupe $Gal(K'_i/K_ik'_i).$ La fibre générique de $\mathcal{Y}_i$ est une variété géométriquement intègre
sur $k'_i.$

Le théorème de \v{C}ebotarev géométrique \cite[Lemma 1.2]{Ekedahl} dit que, pour tout élément
$\sigma\in Gal(K'_i/K_ik'_i)$ et pour presque toute place $v'_i$ de $k'_i,$
le schéma $\mathcal{Y}'_i$ possède des $k'_i(v'_i)$-point $Q(v'_i)$ tels que le Frobenius en $Q(v'_i)$ soit $\sigma,$
on peut supposer que ceci vaut pour toute $v'_i\in\Omega_{k'_i}\setminus S_1\otimes_kk'_i.$

\medskip
\emph{Étape 3: Construction de l'ensemble $S.$}

Comme toute fibre contient une composante irréductible de multiplicité un qui est géométriquement intègre,
l'estimation de Lang-Weil (\textit{cf.} \cite{Lang-Weil}, \cite[Lemme 3.3]{Liang1}) appliquée à une famille plate
dit qu'il existe un ensemble fini $S\supset S_1$ de places de $k$
(et on pose $\mathcal{W}=Spec(O_{k,S})$) contenant toutes les places archimédiennes
tel que
\begin{itemize}
\item[-] si $\theta\in V_0$ est un point fermé et si $\theta$ est un entier en une place
        $w^o\in\Omega_{k(\theta)}\setminus S\otimes_kk(\theta),$
        alors $\widetilde{\mathcal{X}_\theta}^{w^o}$ (la réduction modulo $w^o$ de $X_\theta$)
        possède un $k(\theta)(w^o)$-point hors de $\mathcal{Z}_1.$
\item[-] si $i\in\{1,\ldots,l\}$ et si $v^o_i\in\Omega_{k_i}\setminus S\otimes_kk_i,$ alors $\widetilde{\mathcal{U}_i}^{v^o_i}$
        (la réduction modulo $v^o_i$ de $\mathcal{U}_i$) possède un $k_i(v^o_i)$-point.
\item[-] si $v^o_\infty\in\Omega_k\setminus S,$ alors $\widetilde{\mathcal{U}_\infty}^{v^o_\infty}$
        (la réduction modulo $v^o_\infty$ de $\mathcal{U}_\infty$) possède un $k(v^o_\infty)$-point.
\end{itemize}

\medskip
Maintenant, soit $S'\subset\Omega_k$ un ensemble fini contenant $S.$

\medskip
\emph{Étape 4: Construction des $v_i$ et $\theta_i.$}

D'après le théorème de \v{C}ebotarev, on choisit
des places $v_i(i\in\{1,\ldots,l\})$ deux à deux distinctes, hors de $S',$ et totalement décomposées
dans $k'_i/k.$
Le polynôme $P_i(T)=0$ a alors une racine simple dans $k(v_i)$ se relevant en un élément
$\theta_i^{(1)}$ de $k_{v_i}=\mathbb{A}^1(k_{v_i})$ vérifiant $v_i(P_i(\theta_i^{(1)}))=1.$
On fixe $\theta_i^{(2)}$ un point fermé de $\mathbb{A}^1_{v_i}$ de degré $d-1$ différent de $\theta_i^{(1)}$
(sur un corps $p$-adique, il existe un nombre infini de polynômes irréductibles de degré fixé, \textit{cf.}
\cite{SerreCorpsLoc}, Proposition 17 et son corollaire).
On définit $\theta_i=\theta_i^{(1)}+\theta_i^{(2)}\in Z_0(\mathbb{P}^1_{v_i})$ un zéro-cycle local de degré $d.$

De plus on prend $v_\infty\in\Sigma\setminus S$ distincte des $v_i$ ci-dessus. Enfin, on choisit
$\theta_\infty^{(1)}\in k^*_{v_\infty}$ tel que $v_\infty(1/\theta_\infty)=1$ et construit
$\theta_\infty=\theta_\infty^{(1)}+\theta_\infty^{(2)}\in Z_0(\mathbb{P}^1_{v_\infty})$ de degré $d$ comme précédemment.

\medskip
Dans le reste de la preuve, on vérifiera que les données qu'on a construites satisfont l'assertion.

Soit $\theta\in V_0\subset\mathbb{A}^{1}$ un point fermé de degré $d,$ on suppose qu'il est entier en dehors
de $S'\cup\{v_1,\ldots,v_l,v_\infty\}\cup\Sigma$ et qu'il est suffisamment proche de chaque $\theta_i(i\in\{1,\ldots,l\}\cup\{\infty\})$
comme zéro-cycles locaux.

\medskip
\emph{Étape 5: Classification des places de $k(\theta).$}

Le point fermé $\theta$ est suffisamment proche de $\theta_i(i\in\{1,\ldots,l\}),$ par définition
il existe donc une place $w^0_i$ de $k(\theta)$ au-dessus de $v_i$ telle que les extensions $k(\theta)_{w_i^0}/k_{v_i}$
et $k(\theta)(w_i^0)/k(v_i)$ soient triviales et l'image de $\theta\in k(\theta)=\mathbb{A}^1(k(\theta))$
dans $k(\theta)_{w_i^0}$ soit suffisamment proche
de $\theta_i^{(1)}.$
Donc $w_i^0(P_i(\theta))=v_i(P_i(\theta))=v_i(P_i(\theta_i^{(1)}))=1,$ a fortiori $w_i^0(\theta)\geqslant 0$
car les coefficient de $P_i$ sont $v_i$-entiers. Le point $\theta$ est alors un $v_i$-entier pour $i\in\{1,\ldots,l\}.$
Comme $v_\infty\in\Sigma,$ le point $\theta$ est alors $S'\cup\Sigma$-entier.
De même, il existe une place $w_\infty^0$ de $k(\theta)$ au-dessus de $v_\infty$ tel que
$w_\infty^0(1/\theta)=1.$

On considère la réduction de $\theta\in V_0\subset\mathbb{P}^1$ modulo une place
$w\in\Omega_{k(\theta)}\setminus S'\otimes_kk(\theta),$
il y a trois possibilités:
\begin{itemize}
\item[(a)] elle est dans $\mathcal{V}_0,$ si $w(P_i(\theta))=0$ (a fortiori $w(\theta)\geqslant 0$),
        on note $\Omega_0$ l'ensemble de ces places;
\item[(b)] elle est dans l'un (unique) des  $\widetilde{m_i},$ si $w(P_i(\theta))>0$ (a fortiori $w(\theta)\geqslant 0$),
        on note $\Omega_i$ l'ensemble de ces places, en particulier $w_i^0\in\Omega_i;$
\item[(c)] elle est dans  $\widetilde{\infty}$  si $w(\theta)<0$,
        on note $\Omega_\infty$ l'ensemble de ces places, 
        de plus on a $\Omega_\infty\subset\Sigma\otimes_kk(\theta)$ car $\theta$ est un $S'\cup\Sigma$-entier.
\end{itemize}

Les ensemble $\Omega_0,\Omega_1,\ldots,\Omega_l, \Omega_\infty$ forment une partition de
$\Omega_{k(\theta)}\setminus S'\otimes_kk(\theta).$ L'ensemble $\Omega_i$ est fini pour
$i\in\{1,\ldots,l\}\cup\{\infty\}.$

\medskip
\emph{Étape 6: Existence des points locaux de $X_\theta\cap U_0.$}

On se donne pour chaque $w\in S'\otimes_kk(\theta)$ un $k(\theta)_w$-point $M_w$ de $X_\theta\cap U_0$
avec l'égalité $$\sum_{w\in S'\otimes_kk(\theta)}inv_w(A(M_w))=0$$ pour
$A\in\Lambda\cup E.$

D'après la construction de $S,$ dans les trois cas (a), (b), (c),
le schéma $\widetilde{\mathcal{X}_\theta}^w$ possède un $k(\theta)(w)$-point
$M(w)$ qui se relève par le lemme de Hensel en un $k(\theta)_w$-point $M_w$ de $U_0\cap X_\theta$
pour toute $w\in\Omega_{k(\theta)}\setminus S'\otimes_kk(\theta).$

En général, la famille $\{M_w\}_{w\in\Omega_{k(\theta)}}$
n'est pas orthogonale à $A\in\Lambda\cup E.$ Dans la suite, on va modifier certains $M_w$
pour obtenir l'orthogonalité.

\medskip
\emph{Étape 7: Calculs des $A(M_w).$}

On veut évaluer $A(M_w)$ pour $A\in\Lambda\cup E\subset Br(U_0)$ et $w\in\Omega_{k(\theta)}\setminus S'\otimes_kk(\theta).$
Par construction, le point $M_w$ est dans $U_0$ et sa réduction $M(w)$ est dans
l'un (unique) des $\mathcal{U}_i$ ($i\in\{0,1,\ldots,l\}\cup\{\infty\}$)
selon $w\in\Omega_i.$

D'après \cite[Corollaire 2.4.3]{Harari}, on trouve
\begin{itemize}
\item[-] $inv_w(A(M_w))=0$ si $w\in\Omega_0;$
\item[-] $inv_w(A(M_w))=w(P_i(\theta))\cdot\partial_{A,i,M(w)}=w(P_i(\theta))\cdot\partial_{A,i}(F_{i,M(w)})$
        si $w\in\Omega_i(i\in\{1,\ldots,l\}),$ où $\partial_{A,i,M(w)}$
        est l'évaluation de $\partial_{A,i}\in H^1_{\mbox{\scriptsize\'et}}(\mathcal{U}_i,\mathbb{Q}/\mathbb{Z})$
        en le point $M(w)$ de $\mathcal{U}_i,$ et où $F_{i,M(w)}\in Gal(\mathcal{Y}_i/\mathcal{U}_i)$
        est le Frobenius en $M(w);$
\item[-] $inv_w(A(M_w))=w(1/\theta)\cdot\partial_{A,\infty,M(w)}$ si $w\in\Omega_\infty,$ où $\partial_{A,\infty,M(w)}$
        est l'évaluation de $\partial_{A,\infty}\in H^1_{\mbox{\scriptsize\'et}}(\mathcal{U}_\infty,\mathbb{Q}/\mathbb{Z})$
        en le point $M(w)$ de $\mathcal{U}_\infty.$
\end{itemize}

\medskip
\emph{Étape 8: Calculs pour $w\in\Omega_\infty.$}

Si $A\in\Lambda\subset Br(U),$ comme le point $\infty\in\pi(U),$
le résidu $\partial_{A,\infty}$ est nul dans $H^1(K_\infty,\mathbb{Q}/\mathbb{Z}).$ Si $A\in E,$ il
provient d'un élément de $Br(k(T))$ qui devient nul dans $Br(k'(T))$ par construction. De plus,
la restriction $v$ d'une place $w\in\Omega_\infty$ au corps $k$ appartient dans $\Sigma,$
d'où $k(\theta)_w$ est une extension de $k_v=k'_{v'}$ pour toute place $v'$ de $k'$ au-dessus de $v,$
donc $A$ est nul dans $Br(k(\theta)_w(T)).$ En tout cas, on a $inv_w(A(M_w))=0$ pour $w\in\Omega_\infty$
et $A\in\Lambda\cup E.$

\medskip
\emph{Étape 9: Une observation pour $w\in\Omega_i(i\in\{1,\ldots,l\}).$}

A priori, l'élément $\sum_{w\in\Omega_i}w(P_i(\theta))\cdot F_{i,M(w)}$ se trouve dans le groupe
$Gal(K'_i/K_i),$ on va montrer que cet élément est dans le sous-groupe $Gal(K'_i/K_ik'_i).$

Soit $\rho_i$ un élément quelconque du groupe
$G'_i=H^1(Gal(K_ik'_i/K_i),\mathbb{Q}/\mathbb{Z})\subset G_i,$ il existe alors un élément $A_i$
de $E\subset Br(k(T))$ dont le seul résidu possiblement non nul est en $m_i$ et vaut $\rho_i.$
Dans ce cas, on a $inv_w(A_i(M_w))=0$ pour toute $w\in\Omega_j$ si $1\leqslant j\leqslant l$ et $j\neq i.$
D'autre part, comme la spécialisation de $A_i\in Br(k(X))$ en $X_\theta$ provient de $Br(k(\theta)),$
on trouve l'égalité
$$\sum_{w\in\Omega_{k(\theta)}}inv_w(A_i(M_w))=0.$$
On obtient alors
$$\sum_{w\in\Omega_i}inv_w(A_i(M_w))=0,$$
autrement dit,
$$\sum_{w\in\Omega_i}\rho_i\left(w(P_i(\theta))\cdot F_{i,M(w)}\right)=0.$$
Donc l'élément $\sum_{w\in\Omega_i}w(P_i(\theta))\cdot F_{i,M(w)}$ se trouve dans le
sous-groupe $Gal(K'_i/K_ik'_i).$

\medskip
\emph{Étape 10: Modification des $M_w$ pour $w\in\Omega_i(i\in\{1,\ldots,l\}).$}

On rappelle qu'il existe une place $w^0_i\in\Omega_i$ au-dessus de $v_i$ avec $w^0_i(P_i(\theta))=1,$ de plus
les extensions $k(\theta)_{w_i^0}/k_{v_i}$ et $k(\theta)(w_i^0)/k(v_i)$ sont triviales.
Comme $v_i$ est totalement décomposée dans $k'_i,$ le Frobenius
$F_{i,M(w^0_i)}$ se trouve dans le sous-groupe $Gal(K'_i/K_ik'_i).$

D'après le théorème de \v{C}ebotarev géométrique \cite[Lemma 1.2]{Ekedahl}, il existe
un point $M'(w^0_i)$ de $\mathcal{U}_i$ dont le Frobenius associé $F_{i,M'(w^0_i)}$ est exactement l'élément
$F_{i,M(w^0_i)}-\sum_{w\in\Omega_i}w(P_i(\theta))\cdot F_{i,M(w)}\in Gal(K'_i/K_ik'_i).$
On relève ce point en un $k(\theta)_{w^0_i}$-point $M'_{w^0_i}$ de $X_\theta\cap U_0.$

Pour chaque $i\in\{1,\ldots,l\},$ on remplace $M_{w^0_i}$ par $M'_{w^0_i}$ et on garde $M_w$ pour
$w\in\Omega_i\setminus\{w^0_i\}.$
On vérifie alors que
$$\sum_{w\in\Omega_i}inv_w(A(M_w))=0$$
pour $A\in\Lambda\cup E$ et $i\in\{1,\ldots,l\},$ donc
$$\sum_{w\in\Omega_{k(\theta)}}inv_w(A(M_w))=0$$
pour tout $A\in\Lambda\cup E.$ Ceci termine la preuve.
\end{proof}

\begin{thm}\label{thm10}
Soit $\pi:X\To\mathbb{P}^1$ une fibration.
On suppose que toute fibre contient une composante irréductible de multiplicité un qui est géométriquement intègre.
On suppose que $Br(X_{\bar{\eta}})$ est fini et que $Pic X_{\bar{\eta}}$ est sans torsion.

Supposons qu'il existe un sous-ensemble hilbertien généralisé $\textsf{Hil}$ de $\mathbb{P}^1$
tel que pour tout $\theta\in \textsf{Hil},$ respectivement,

(i) l'obstruction de Brauer-Manin au principe de Hasse est la seule pour
les points rationnels ou pour les zéro-cycles de degré $1$ sur $X_\theta;$

(ii) l'obstruction de Brauer-Manin est la seule à l'approximation faible pour
les points rationnels ou pour les
zéro-cycles de degré $1$ sur $X_\theta;$

(iii) on suppose (ii), et de plus, (H CH0) l'application $CH_0(X_v)\To CH_0(\mathbb{P}^1_v)$ est
un isomorphisme pour presque toute place $v\in\Omega,$
par exemple, la fibre générique $X_\eta$ est une $k(\mathbb{P}^1)$-variété rationnellement connexe.

Alors, respectivement,

(i) l'obstruction de Brauer-Manin est la seule
au principe de Hasse pour les zéro-cycles de degré $1$ sur $X;$

(ii) l'obstruction de Brauer-Manin est la seule
à l'approximation faible pour les
zéro-cycles de degré $1$ sur $X;$

(iii) l'obstruction de Brauer-Manin est la seule
à l'approximation forte pour les
zéro-cycles de degré $1$ sur $X.$
\end{thm}

\begin{proof}

Soit $\Lambda=\{A_1,\ldots,A_r\}\subset Br(k(X))$ un ensemble fini qui engendre le groupe
$Br(X_\eta)/Br(K).$
D'après la proposition \ref{Br-specialisation}, il existe un sous-ensemble hilbertien généralisé
$\textsf{Hil}'\subset \textsf{Hil}$ tel que
pour tout point fermé $\theta\in \textsf{Hil}',$ la flèche de spécialisation
$$sp_\theta:\frac{Br(X_\eta)}{Br(K)}\To \frac{Br(X_\theta)}{Br(k(\theta))}$$
soit un isomorphisme.
Les éléments $A_\theta=sp_\theta(A)(A\in\Lambda)$ engendrent le groupe
$Br(X_\theta)/Br(k(\theta)).$

Soit $U$ un ouvert de $X$ tel que $\Lambda\subset Br(U)$ et tel que les fibres $X_\theta$ qui rencontrent
$U$ soient lisses et géométriquement intègres.
On pose $Z=X\setminus U=Z_1\cup Z_2$ avec $Z_2=\bigsqcup_{i=1}^lX_{m_i}$ comme dans la proposition
\ref{lem-complicated}.
On fixe un $k$-point $\infty$ de $\mathbb{P}^1$ différent de chaque $m_i(i\in\{1,\ldots,l\}),$
tel que la fibre $X_\infty$ soit lisse et géométriquement intègre, on note $U_0=U\setminus X_\infty$ et $V_0=\pi(U_0).$
On pose $D=\mathbb{P}^1\setminus V_0.$
En restreignant $\textsf{Hil}'$ si nécessaire, on peut supposer que $D\cap \textsf{Hil}'=\emptyset.$

Soit $S$ un ensemble fini de places de $k.$
La proposition \ref{lem-complicated}  donne un ensemble fini $S_1$ de places de $k$ contenant $S,$
en ensemble infini $\Sigma\subset\Omega_k,$
et un ensemble fini
$E\subset Br(K)\cap Br(U_0)$ possédant la propriété décrite dans \ref{lem-complicated}.
On choisit $N_0$ un point fermé de $U_0$ et $a$ un entier positif
tel que les éléments de $\Lambda\cup E$ soient tués par $a.$ On note $n_0$ le degré du point $N_0.$

On part d'une famille de zéro-cycles $z_v\in Z_0(X_v)$ de degré $1$ telle que
$$\sum_{v\in\Omega}inv_v( \langle b,z_v \rangle _v)=0\mbox{ pour tout }b\in Br(X).$$
On peut supposer que les $z_v$ sont à support dans $U_0$ par le lemme de déplacement \ref{deplacement}.
Le lemme formel (Lemme \ref{lemme formel}) dit qu'il existe un ensemble fini $S_2$ de places de $k$ contenant $S_1$
et $z'_v\in Z_0({U_0}_v)(v\in S_2)$ de degré $1$ avec $z'_v=z_v(v\in S_1)$
tels que
$$\sum_{v\in S_2}inv_v( \langle A,z'_v \rangle _v)=0\mbox{ pour tout }A\in \Lambda\cup E.$$

Pour $v\in S_2,$ on écrit $z'_v=z^+_v-z^-_v$ où $z^+_v$ et $z^-_v$ sont des zéro-cycles effectifs à supports disjoints dans $U_0$.
On pose $z^1_v=z'_v+an_0z^-_v=z^+_v+(an_0-1)z^-_v,$ alors $deg(z^1_v)\equiv 1(mod\mbox{ }an_0),$
$ \langle A,z'_v \rangle _v= \langle A,z^1_v \rangle _v$ pour tout $A\in\Lambda\cup E.$
On a aussi que $ \langle A,aN_0 \rangle _v=0$ car les $A\in\Lambda\cup E$ sont tués par $a.$
On ajoute un multiple positif convenable de $aN_0$ à chaque $z^1_v(v\in S_2)$ et trouve un zéro-cycle
$z^2_v$ de même degré $d\equiv 1(mod\mbox{ }an_0)$ pour $v\in S_2.$  On a
$ \langle A,z^1_v \rangle _v= \langle A,z^2_v \rangle _v$ pour $A\in \Lambda\cup E.$
Donc $$\sum_{v\in S_2}inv_v( \langle A,z^2_v \rangle _v)=0\mbox{ pour tout }A\in \Lambda\cup E.$$

D'après le lemme de déplacement \ref{deplacement2}, il existe, pour $v\in S_2,$ zéro-cycle effectif $z^3_v$ de degré $d$
à support dans $U_0$ tel que $\pi_*(z^3_v)$ soit séparable et tel que $z^3_v$ soit suffisamment proche de $z^2_v.$
D'où $$\sum_{v\in S_2}inv_v( \langle A,z^3_v \rangle _v)=0\mbox{ pour tout }A\in \Lambda\cup E.$$

Avec l'entier $d,$ la proposition \ref{lem-complicated}  donne
les places $v_1,\ldots,v_l,v_\infty\in\Omega\setminus S_2$ avec $v_\infty\in\Sigma$
et le zéro-cycle $\theta_i\in Z_0(\mathbb{P}^1_{v_i})(i\in\{1,\ldots,l\}\cup\{\infty\})$
effectif satisfaisant la propriété décrite dans \ref{lem-complicated}.

On prend $v_0\in\Sigma\setminus(S_2\cup\{v_1,\ldots,v_l\}\cup\{v_\infty\})$ une place non-archimédienne.
D'après le lemme \ref{simplified} appliqué à $S_2\cup\{v_1,\ldots,v_l\}\cup\{v_\infty\}$ et $v_0,$
on trouve $\theta\in\textsf{Hil}'\subset V_0$ un point fermé de degré $d,$
tel que $\theta$ soit suffisamment proche de $\pi_*(z^3_v)(v\in S_2)$ et suffisamment proche
de $\theta_i(i\in\{1,\ldots,l\}\cup\{\infty\}),$ de plus $\theta$ est entier en dehors de
$S_2\cup\{v_0,v_1,\ldots,v_l,v_\infty\}.$ En particulier, le point $\theta\in V_0$ est un
$S_2\cup\{v_1,\ldots,v_l,v_\infty\}\cup\Sigma$-entier.

Comme $\theta\times_{\mathbb{P}^1}{\mathbb{P}_v^1}=\bigsqcup_{w\mid v,w\in\Omega_{k(\theta)}}Spec(k(\theta)_w)$
pour $v\in\Omega,$ l'image de $\theta$ dans
$Z_0({\mathbb{P}^1_v})$ s'écrit comme $\theta_v=\sum_{w\mid v,w\in\Omega_{k(\theta)}}P_w$ où $P_w=Spec(k(\theta)_w)$ est un
point fermé de ${\mathbb{P}^1_v}$ de corps résiduel $k(\theta)_w.$
Pour $v\in S_2,$ $\theta_v$ est suffisamment proche de $\pi_*(z^3_v),$ où le zéro-cycle effectif séparable
$\pi_*(z^3_v)$ est de la forme $\sum_{w\mid v,w\in\Omega_{k(\theta)}}Q_w$ avec
les $Q_w$ deux à deux distincts. Alors $k(\theta)_w=k_v(P_w)=k_v(Q_w),$
$P_w$ est suffisamment proche de $Q_w\in \mathbb{P}^1_v(k(\theta)_w).$ D'où on sait que
$z^3_v=\sum_{w\mid v,w\in\Omega_{k(\theta)}}M^0_w$ avec $k_v(M^0_w)=k(\theta)_w$ et $M^0_w\in X_v(k(\theta)_w)$
se trouve dans la fibre au-dessus du point fermé $Q_w.$
Le théorème des fonctions implicites implique qu'il existe un $k(\theta)_w$-point lisse $M_w$ de $X_\theta\cap U_0$
suffisamment proche de $M^0_w$ pour toute $w\in S_2\otimes_kk(\theta).$

On a donc, pour $A\in \Lambda\cup E,$


\begin{eqnarray*}
   \sum_{w\in S_2\otimes_kk(\theta)}inv_w(A(M_w))&=&\sum_{w\in S_2\otimes_kk(\theta)}inv_w(A(M^0_w)) \\
  =\sum_{v\in S_2}\sum_{w\mid v}inv_w(A(M^0_w))\mbox{ }\mbox{ }&=&\sum_{v\in S_2}\sum_{w\mid v}inv_v(cores_{k(\theta)_w/k_v}(A(M^0_w)))\\
  =\sum_{v\in S_2}\sum_{w\mid v}inv_v( \langle A,M^0_w \rangle _v)&=&\sum_{v\in S_2}inv_v( \langle A,z^3_v \rangle _v)\mbox{ }=\mbox{ }0\\
\end{eqnarray*}

La proposition \ref{lem-complicated}  donne des $k(\theta)_w$-points lisses $M_w$ de $X_\theta\cap U_0$
pour $w\in \Omega_{k(\theta)}\setminus S_2\otimes_kk(\theta)$ tels que
$$\sum_{w\in\Omega_{k(\theta)}}inv_w(A(M_w))=0\mbox{ pour tout }A\in\Lambda\cup E.$$

Les points rationnels $M_w\in(X_\theta\cap U_0)(k(\theta)_w)$ définissent une famille de zéro-cycles de degré $1$
de $X_\theta$ vue comme une $k(\theta)$-variété. On a
$$\sum_{w\in\Omega_{k(\theta)}}inv_w( \langle A_\theta,M_w \rangle _w)=\sum_{w\in\Omega_{k(\theta)}}inv_w(A(M_w))=0\mbox{ pour tout }A\in\Lambda\cup E.$$

Comme $\theta\in \textsf{Hil}',$ les $A_\theta$ engendrent le
groupe $Br(X_\theta)/Br(k(\theta)),$
on trouve $\{M_w\}_{w\in\Omega_{k(\theta)}}\bot Br(X_\theta).$
Comme $\theta\in \textsf{Hil},$ la fibre $X_\theta$ satisfait l'hypothèse (i) (resp. (ii),(iii)),
il existe alors un zéro-cycle global $z'\in Z_0(X_\theta)$ de degré $1$ (sur $k(\theta)$),
donc $z'\in Z_0(X)$ est de degré $d\equiv 1(mod\mbox{ }an_0),$
alors $z=z'-\frac{d-1}{n_0}N_0$ est un zéro-cycle global de degré $1$ sur $X.$
Pour montrer l'approximation faible/forte pour les zéro-cycles, on fixe $m$ un entier positif au tout début.
De plus, on peut choisir $a$ un multiple de $m.$
On prend $S$ le sous-ensemble fini considéré pour l'approximation (resp. un sous-ensemble fini qui
contient toutes les places telles que $CH_0(X_v)\To CH_0(\mathbb{P}^1_v)$ ne soit pas injective).
Ensuite on fait fonctionner l'argument ci-dessus,
à l'aide du lemme 1.8 de Wittenberg \cite{Wittenberg}, on obtient un zéro-cycle global
$z=z_m$ ayant la même image que $z_v$ dans $CH_0(X_v)/m$ pour $v\in S.$
\end{proof}


\subsection{Cas où la fibre générique admet un zéro-cycle de degré $1$}\label{subsection1}

Pour une fibration $X\to\P^1,$ si l'on suppose que la fibre générique $X_\eta/k(\mathbb{P}^1)$ admet
un zéro-cycle de degré $1,$
le principe de Hasse pour les zéro-cycles de degré $1$ vaut pour $X.$
On s'intéresse seulement à l'approximation faible.

\begin{thm}\label{thm15}
Soient $\pi:X\To\mathbb{P}^1$ une fibration et
$\textsf{Hil}$ un sous-ensemble hilbertien généralisé de $\P^1.$
On suppose que

- l'indice de la fibre générique $ind(X_\eta/K)=1,$ où $K=k(\P^1),$

- le groupe $Br(X_{\bar{\eta}})$ est fini, le groupe $PicX_{\bar{\eta}}$ est sans torsion,

- pour tout point fermé $\theta\in \textsf{Hil},$
l'obstruction de Brauer-Manin est la seule à l'approximation faible
pour les zéro-cycles de degré $1$ sur la fibre $X_\theta.$

Alors l'obstruction de Brauer-Manin est la seule à l'approximation faible
pour les zéro-cycles de degré $1$ sur $X.$

Si de plus on suppose que
\begin{itemize}
\item[(H CH0)]
l'application $CH_0(X_v)\To CH_0(\mathbb{P}^1_v)$ est un isomorphisme pour presque toute place $v\in\Omega_k,$
par exemple la fibre générique $X_\eta$ est une $k(\mathbb{P}^1)$-variété  rationnellement connexe,
\end{itemize}
alors l'obstruction de Brauer-Manin est la seule à l'approximation forte
pour les zéro-cycles de degré $1$ sur $X.$
\end{thm}

\begin{proof}
On suit la stratégie de la preuve du théorème 4.3.1 de Harari \cite{Harari}.

D'après la proposition \ref{Br-specialisation}, il existe un sous-ensemble hilbertien généralisé
$\textsf{Hil}_0$ tel que pour tout $\theta\in\textsf{Hil}_0$ l'application de spécialisation soit un isomorphisme
de groupes finis. On peut supposer que les fibres $X_\theta(\theta\in\textsf{Hil}_0)$ sont lisses géométriquement intègres.
Soit $\Lambda=\{A_1,\ldots,A_r\}\subset Br(k(X))$ un ensemble de représentants du groupe fini $Br(X_\eta)/Br(k(\mathbb{P}^1)).$
Les éléments $A_\theta=sp_\theta(A)\in Br(k(\theta)(X_\theta))(A\in\Lambda)$ engendrent
$Br(X_\theta)/Br(k(\theta))$ pour tout $\theta\in \textsf{Hil}_0.$

Par l'hypothèse, il existe un zéro-cycle $z_0$ de degré $1$ sur $X_\eta,$ on écrit
$z_0=\sum_jn_jR_j(\eta)$ avec $K_j=K(R_j(\eta))$ le corps résiduel du point $R_j(\eta).$
On pose $d_j=[K(R_j(\eta)):K].$ Il existe alors un ouvert non vide $U$ de $\mathbb{P}^1,$
un morphisme étale fini $Z_j\To U$ de degré $d_j,$ où $Z_j$ est une sous-variété fermé intègre de
$X\times_{\mathbb{P}^1}U$ de corps des fonctions $K_j$ pour chaque $j.$

Pour chaque $i,$ on pose $A'_i=\sum_jn_jcores_{K_j/K}(A_i(R_j(\eta)))\in Br(K).$ En
remplaçant $A_i$ par $A_i-A'_i,$ on peut supposer que $ \langle A_i,\sum_jn_jR_j(\eta) \rangle _K=0.$
Soit $X_0$ un ouvert non vide de $X$ tel que $\Lambda\subset Br(X_0).$
En restreignant $U,$ on peut supposer que $\pi_{|X_0}:X_0\To U$ est surjectif.
On fixe $N_0$ un point fermé de $X_0$ de degré $n_0.$
On choisit  $a$ un entier positif
tel que les éléments de $\Lambda$ soient tués par $a$ et tel que $a$ est un multiple de $m$ (si l'on est en train de considérer
l'approximation concernant l'image dans $CH_0(-)/m$).

On part d'une famille de zéro-cycles de degré $1,$ $\{z_v\}\bot Br(X),$
on va l'approximer aux places dans un sous-ensemble fini $S\subset\Omega_k.$
Pour l'approximation forte,
on doit prendre $S$ un sous-ensemble fini qui
contient toutes les places telles que $CH_0(X_v)\To CH_0(\mathbb{P}^1_v)$ ne soit pas injective.
On peut supposer que
les supports des $z_v$ sont dans $X_0$ par le lemme de déplacement \ref{deplacement}.

Le lemme formel (Lemme \ref{lemme formel}) dit qu'il existe un ensemble fini $S'$ de places de $k$ contenant $S$
et il existe $z'_v\in Z_0(X_{0v})(v\in S')$ de degré $1$ avec $z'_v=z_v(\forall v\in S)$
tels que
$$\sum_{v\in S'}inv_v( \langle A_i,z'_v \rangle _v)=0\mbox{ pour tout }A_i\in \Lambda.$$

Pour $v\in S',$ on écrit $z'_v=z^+_v-z^-_v$ où $z^+_v$ et $z^-_v$
sont des zéro-cycles effectifs à supports distincts dans $X_0$.
On pose $z^1_v=z'_v+an_0z^-_v=z^+_v+(an_0-1)z^-_v,$ alors $deg(z^1_v)\equiv 1(mod\mbox{ }an_0),$
$ \langle A_i,z'_v \rangle _v= \langle A_i,z^1_v \rangle _v$ pour tout $A_i\in\Lambda.$
On a aussi que $ \langle A_i,aN_0 \rangle _v=0$ car les $A_i\in\Lambda$ sont tués par $a.$
On ajoute un multiple positif convenable de $aN_0$ à chaque $z^1_v(v\in S')$ et trouve
$z^2_v$ de même degré $\delta\equiv 1(mod\mbox{ }an_0)$ pour $v\in S'.$  On a
$ \langle A_i,z^1_v \rangle _v= \langle A_i,z^2_v \rangle _v$ pour $A_i\in \Lambda.$
Donc $$\sum_{v\in S'}inv_v( \langle A_i,z^2_v \rangle _v)=0\mbox{ pour tout }A_i\in \Lambda.$$

D'après le lemme de déplacement \ref{deplacement2}, il existe, pour $v\in S',$ un zéro-cycle effectif $z^3_v$ de degré $\delta$
à support dans $X_0$ tel que $\pi_*(z^3_v)$ soit séparable et tel que $z^3_v$ soit suffisamment proche de $z^2_v.$
D'où $$\sum_{v\in S'}inv_v( \langle A_i,z^3_v \rangle _v)=0\mbox{ pour tout }A_i\in \Lambda.$$

Par le lemme \ref{simplified}, il existe un point fermé $\theta\in \textsf{Hil}',$  où $\textsf{Hil}'$
est un sous-ensemble hilbertien généralisé contenu dans $\textsf{Hil}\cap\textsf{Hil}_0\cap U,$ tel que
$\theta$ soit suffisamment proche
de $\pi_*(z^3_v)$ pour toute $v\in S'.$

Comme $\theta\times_{\mathbb{P}^1}{\mathbb{P}_v^1}=\bigsqcup_{w\mid v,w\in\Omega_{k(\theta)}}Spec(k(\theta)_w)$
pour $v\in\Omega,$ l'image de $\theta$ dans
$Z_0({\mathbb{P}^1_v})$ s'écrit comme $\theta_v=\sum_{w\mid v,w\in\Omega_{k(\theta)}}P_w$ où $P_w=Spec(k(\theta)_w)$ est un
point fermé de ${\mathbb{P}^1_v}$ de corps résiduel $k(\theta)_w.$
Pour $v\in S',$ $\theta_v$ est suffisamment proche de $\pi_*(z^3_v),$ où
le zéro-cycle effectif séparable $\pi_*(z^3_v)$ est de la forme $\sum_{w\mid v,w\in\Omega_{k(\theta)}}Q_w$ avec
les $Q_w$ deux à deux distincts. Alors $k(\theta)_w=k_v(P_w)=k_v(Q_w),$
$P_w$ est suffisamment proche de $Q_w\in \mathbb{P}^1_v(k(\theta)_w).$ D'où on sait que
$z^3_v=\sum_{w\mid v,w\in\Omega_{k(\theta)}}M^0_w$ avec $k_v(M^0_w)=k(\theta)_w$ et $M^0_w\in X_v(k(\theta)_w)$ se trouve
dans la fibre au-dessus
du point fermé $Q_w.$
Le théorème des fonctions implicites implique qu'il existe un $k(\theta)_w$-point lisse $M_w$ de $X_\theta\cap X_0$ suffisamment
proche de $M^0_w$ pour toute $w\in S'\otimes_kk(\theta).$

On a donc, pour $A_i\in \Lambda,$


\begin{eqnarray*}
&&\sum_{w\in S'\otimes_kk(\theta)}inv_w( \langle A_i,M_w \rangle _w)\mbox{ }\mbox{ }\mbox{ }\mbox{ }\mbox{ }\mbox{ }\mbox{ }\mbox{ }\mbox{ }\mbox{ }\mbox{ }\mbox{ }\mbox{ }=\sum_{w\in S'\otimes_kk(\theta)}inv_w(A_i(M_w)) \\
&=&\sum_{w\in S'\otimes_kk(\theta)}inv_w(A_i(M^0_w))\mbox{ }\mbox{ }\mbox{ }\mbox{ }\mbox{ }\mbox{ }\mbox{ }\mbox{ }\mbox{ }\mbox{ }\mbox{ }\mbox{ }\mbox{ }\mbox{ }\mbox{ }\mbox{ }=\sum_{v\in S'}\sum_{w\mid v}inv_w(A_i(M^0_w))\\
&=&\sum_{v\in S'}\sum_{w\mid v}inv_v(cores_{k(\theta)_w/k_v}(A_i(M^0_w)))\mbox{ }=\sum_{v\in S'}\sum_{w\mid v}inv_v( \langle A_i,M^0_w \rangle _v)\\
&=&\sum_{v\in S'}inv_v( \langle A_i,z^3_v \rangle _v)\mbox{ }\mbox{ }\mbox{ }\mbox{ }\mbox{ }\mbox{ }\mbox{ }\mbox{ }\mbox{ }\mbox{ }\mbox{ }\mbox{ }\mbox{ }\mbox{ }\mbox{ }\mbox{ }\mbox{ }\mbox{ }\mbox{ }\mbox{ }\mbox{ }\mbox{ }\mbox{ }\mbox{ }=\mbox{ }0\\
\end{eqnarray*}

- Pour $v\in S',$ on pose  $z_v^4=\sum_{w\in \Omega_{k(\theta)},w|v}M_w,$ c'est un zéro-cycle
effectif de degré $1$ sur $X_\theta\times_kk_v.$
Vu comme un zéro-cycle (de degré $\delta$) de $X_v$ il est suffisamment proche de $z_v^3.$

- Pour $v\notin S'$ on pose
$z_v^4=\sum_jn_jR_j(\theta_v):=\sum_jn_jR_j(\sum_{w\in\Omega_{k(\theta)},w|v}P_w)=\sum_{w\in\Omega_{k(\theta)},w|v}\sum_jn_jR_j(P_w),$
où la spécialisation $R_j(P_w)=Z_j\times_UP_w$ est un zéro-cycle effectif de degré $d_j$ dans la fibre
${X_v}_{P_w}.$ Alors $M_w=\sum_jn_jR_j(P_w)$ est  un zéro-cycle effectif de degré $1$ dans la fibre
${X_v}_{P_w}.$
Comme $ \langle A_i,\sum_jn_jR_j(\eta) \rangle _{K}=0,$ on a $ \langle A_i,M_w \rangle _w=0,$ donc
$$\sum_{w\in (\Omega\setminus S')\otimes_kk(\theta)}inv_w( \langle A_i,M_w \rangle _w)=0.$$ On arrive à
$$\sum_{w\in \Omega_{k(\theta)}}inv_w( \langle A_i,M_w \rangle _w)=0,\mbox{ pour tout }A_i\in\Lambda.$$

Pour chaque $w\in\Omega_{k(\theta)},$ $M_w$ est un zéro-cycle effectif local de degré $1$ sur la $k(\theta)$-variété
$X_\theta,$ donc

$$\sum_{w\in \Omega_{k(\theta)}}inv_w( \langle A_{i\theta},M_w \rangle _w)=0,\mbox{ pour tout }A_i\in\Lambda.$$

Comme $\theta\in\textsf{Hil}_0,$ les $A_{i\theta}$ engendrent le
groupe $Br(X_\theta)/Br(k(\theta)),$
on a donc $\{M_w\}_{w\in\Omega_{k(\theta)}}\bot Br(X_\theta).$
Comme $\theta\in \textsf{Hil},$
l'obstruction de Brauer-Manin est la seule à l'approximation faible
pour les zéro-cycles de degré $1$ sur la fibre $X_\theta,$
il existe alors un zéro-cycle global $z_m'\in Z_0(X_\theta)$ de degré $1$ (sur $k(\theta)$),
tel que pour toute $w|v\in S',$ $z_m'$ et $z^4_v$ ont la même image dans
$CH_0(X_{vP_w})/m.$
Donc $z_m'\in Z_0(X)$ est de degré $\delta\equiv 1(mod\mbox{ }an_0),$
alors $z_m=z_m'-\frac{\delta-1}{n_0}N_0$ est un zéro-cycle global de degré $1.$
En notant que $\frac{\delta-1}{n_0}$ est un multiple
de $a$ (\textit{a fortiori} de $m$), à l'aide du lemme 1.8 de Wittenberg \cite{Wittenberg}, on
vérifie que les zéro-cycles $z=z_m$ et $z_v$ ont la même image
dans $CH_0(X_v)/m$ pour toute $v\in S.$
\end{proof}


\subsection{Quelques remarques sur cette section}

\begin{rem}\label{rem-B}
Comme indiqué par Harari dans \cite{Harari3}, dans cette section, on peut considérer l'obstruction
de Brauer-Manin associée à un certain sous-groupe, ce sera plus flexible.
Soit $B$ un sous-groupe de $Br(X_\eta)$ contenant l'image de $Br(k(\P^1))$ tel que
$B/Br(k(\P^1))$ soit fini. Pour presque tout point fermé $\theta$ de $\P^1,$ la flèche de spécialisation
$sp_\theta^B:B/Br(k(\P^1))\to Br(X_\theta)/Br(k(\theta))$ est bien définie, on pose $B_\theta$ l'image de $B.$
Si sur la fibre $X_\theta(\theta\in \textsf{Hil})$
l'obstruction de Brauer-Manin associée au sous-groupe $B_\theta$ est la seule, on peut
conclure sans difficulté que l'obstruction associée au sous-groupe $B\cap Br(X)$ est la seule sur $X.$
En particulier, si le groupe $Br(X_{\bar{\eta}})$ est fini et le groupe $PicX_{\bar{\eta}}$ est sans torsion,
le groupe $Br(X_\eta)/Br(k(\P^1))$ est alors un groupe fini, en
prenant $B=Br(X_\eta),$ on rentre (via la proposition \ref{Br-specialisation})
dans le cadre des théorèmes principaux \ref{thm10}, \ref{thm15}.
\end{rem}

\begin{rem}
Concernant l'hypothèse arithmétique supposée dans \S \ref{P1},
en pratique, on vérifie souvent que l'obstruction de Brauer-Manin est la seule sur la fibre $X_\theta$ pour tout $\theta$ dans
un ouvert non vide (au lieu d'un certain sous-ensemble hilbertien généralisé) de la base. Cependant, même si l'on
suppose cette hypothèse pour les $\theta$ dans un ouvert non vide,
les preuves ne deviennent pas plus simples, parce que la proposition \ref{Br-specialisation}
est valable seulement pour les $\theta\in \textsf{Hil}.$
\end{rem}


\section{Fibrations au-dessus de $\mathbb{P}^n$}\label{Pn}

Le théorème \ref{thm10} nous permet de démontrer par récurrence des résultats sur une fibration au-dessus
de l'espace projectif $\mathbb{P}^n$ ou une base un peu plus générale.

D'abord, on fait la remarque suivante, qui est souvent utilisée dans cette section:
comme le groupe de Chow des zéro-cycles et le groupe de Brauer sont des invariants birationnels pour les
variétés propres lisses (\cite{Fulton}, Exemple 16.1.11, et \cite{Br}, III \S 7), d'après la fonctorialité de l'accouplement
de Brauer-Manin, la propriété que l'obstruction de Brauer-Manin est la seule au principe de Hasse
(resp. à l'approximation faible/forte)
pour les zéro-cycles de degré $1$ est aussi birationnellement invariante pour les variétés propres lisses.

\medskip
\noindent\textbf{Fibrations au-dessus de $\mathbb{P}^n$}

On considère une variété $X$ projective lisse et géométriquement intègre sur un corps de nombre $k.$
On suppose
que $X$ admet un morphisme dominant $\pi:X\to \mathbb{P}^{n_1}\times\ldots\times\mathbb{P}^{n_r}$
à fibre générique $X_\eta$ (lisse)
géométriquement intègre. On définit le \textit{lieu dégénéré} $D$ comme
l'ensemble des points (schématiques) $P$ au-dessus desquels la fibre
$X_P$ n'est pas géométriquement intègre.
Comme le point générique $\eta$ de $X$ n'appartient pas à l'ensemble constructible $D,$ 
la codimension $codim(D)$ de l'adhérence de Zariski
$\overline{D}$ dans $\mathbb{P}^{n_1}\times\ldots\times\mathbb{P}^{n_r}$ est au moins $1.$
Dans le théorème \ref{Pn-codim2} ci-dessous on va supposer que $codim(D)\geqslant2,$ cependant,
dans les théorèmes \ref{Pn-section}, \ref{Pn-codim1}, d'autres hypothèses sur les fibres sont faites.

\subsection{Cas où les fibres sont géométriquement intègres}
Dans le théorème suivant, le cas où $n=1$ est exactement le théorème \ref{thm10}.

\begin{thm}\label{Pn-codim2}
Soit $\pi:X\to \mathbb{P}^{n_1}\times\ldots\times\mathbb{P}^{n_r}$ une fibration sur
un corps de nombres $k$ avec $X$ une variété projective
lisse et géométriquement intègre.
On suppose que $Br(X_{\bar{\eta}})$ est fini et $Pic(X_{\bar{\eta}})$ est sans
torsion, où $X_{\eta}$ est la fibre générique de $\pi.$

Supposons que $codim(D)\geqslant2.$

On fait l'hypothèse qu'il existe un ouvert non vide $U$ de $\mathbb{P}^{n_1}\times\ldots\times\mathbb{P}^{n_r}$
tel que pour tout point fermé $\theta\in U$ on ait respectivement
\begin{itemize}
\item[(i)]
l'obstruction de Brauer-Manin est la seule au principe de Hasse pour les points rationnels ou pour
les zéro-cycles de degré $1$ sur $X_\theta;$
\item[(ii)]
l'obstruction de Brauer-Manin est la seule à l'approximation faible pour les points rationnels ou pour
les zéro-cycles de degré $1$ sur $X_\theta;$
\item[(iii)]
la condition (ii) et de plus $X_{\eta}$ est rationnellement connexe.
\end{itemize}

Alors, pour les zéro-cycles de degré $1$ sur $X,$ l'obstruction de Brauer-Manin est la seule
\begin{itemize}
\item[(i)]
au principe de Hasse;
\item[(ii)]
à l'approximation faible;
\item[(iii)]
à l'approximation forte.
\end{itemize}
\end{thm}

\begin{proof}

\emph{Étape 1.}
Le cas où la base est $\mathbb{P}^1$ a été montré dans le théorème \ref{thm10}.

\emph{Étape 2.}
Montrons le théorème avec la base $\mathbb{P}^n$
par récurrence sur $n,$ cette idée a été utilisée par Harari dans \cite{Harari3},
et par Wittenberg dans \cite{WittenbergLNM}, pour la question sur les points rationnels.
On utilise une variante de cette méthode, dans l'argument suivant le rôle du théorème \ref{thm10}
y est crucial.

À partir de maintenant, on fixe un entier $n\geqslant2,$ et on admet le théorème pour $\mathbb{P}^{n-1}.$
En restreignant l'ouvert $U$ dans l'hypothèse du théorème si nécessaire, on peut
supposer que toute fibre $X_\theta$ au-dessus d'un point schématique $\theta$ de $U$ est non vide,
projective, lisse, géométriquement intègre, et rationnellement connexe si $X_{\eta_n}$ est supposée rationnellement
connexe (\cite{Kollar}, 3.11). D'après la discussion au début de \S \ref{specialisation}, on peut aussi supposer
que $Br(X_{\bar{\theta}})$ est fini et $Pic(X_{\bar{\theta}})$ est sans
torsion. De plus, on peut supposer que $\overline{D}\cap U=\emptyset.$
On fixe un sous-espace linéaire $O\simeq\mathbb{P}^{n-2}_k$ de $\mathbb{P}^n$ tel que son point générique soit
dans $U$ (par convention, $\mathbb{P}^0$ est un point $k$-rationnel),
et on fixe un sous-espace linéaire $L\simeq\mathbb{P}^1$
de $\mathbb{P}^n$ disjoint de $O.$ On définit l'application rationnelle $g':\mathbb{P}^n\dashrightarrow L\simeq\mathbb{P}^1$
comme la projection de centre $O$ dans $\mathbb{P}^n,$ d'où on obtient un morphisme $g:\Delta\to\mathbb{P}^1$
tel que $g=g'\circ\epsilon$ avec $\epsilon:\Delta\to\mathbb{P}^n$ l'éclatement de $\mathbb{P}^n$ de centre $O.$
La variété $X'=X\times_{\mathbb{P}^n}\Delta$ est projective, géométriquement intègre, et
birationnelle à la variété $X.$
Si $n=2,$ la variété $X'$ est lisse, car les lieux de singularité de $\pi$ et de $\epsilon$ ne se rencontrent pas,
mais ce n'est pas toujours le cas si $n>2.$
Afin de simplifier les notations, on suppose pour l'instant que $X'$ est une variété lisse, à la fin de l'étape 2, on explique
comment completer l'argument sans cette hypothèse supplémentaire.
D'après la remarque au début de cette section, on se ramène à démontrer
que l'obstruction de Brauer-Manin est la seule pour les zéro-cycles de degré $1$ sur $X'.$

On va appliquer le théorème \ref{thm10} à la fibration $g\circ\pi':X'\to\mathbb{P}^1,$ où
$\pi':X'\to \Delta$ est la projection naturelle,
on vérifie toutes les hypothèses comme suit.
Pour tout point schématique
$\theta\in\mathbb{P}^1,$ la fibre $\Delta_\theta$ est isomorphe à l'espace projectif $\mathbb{P}^{n-1}$ sur
$k(\theta),$ son point générique $\eta(\Delta_\theta)$ est contenu dans l'ouvert $\epsilon^{-1}(U)\subset \Delta$
par construction. La fibre de $\pi':X'\to\Delta$ au-dessus du point $\eta(\Delta_\theta)$ est alors
projective, lisse, géométriquement intègre et satisfait les conditions sur le groupe de Picard et sur le
groupe de Brauer.
Le morphisme $\pi'_\theta:X'_\theta\to\Delta_\theta\simeq\mathbb{P}^{n-1}_{k(\theta)}$ est donc dominant à
fibre générique géométriquement intègre (et rationnellement connexe si $X_{\eta_n}$ est supposée rationnellement connexe),
il satisfait
aussi les conditions sur le groupe de Picard et sur le groupe de Brauer.
Pour tout point schématique $\theta,$
la $k(\theta)$-variété $X'_\theta$ elle-même est alors
géométriquement intègre sur $k(\theta)$ (car $k(\theta)$ est algébriquement fermé dans
$k(\theta)(\Delta_\theta)$ qui est algébriquement fermé dans $k(\theta)(X'_\theta)$);
et elle est rationnellement connexe d'après un résultat de Graber/Harris/Starr \cite[Corollaire 1.3]{GHS}
une fois que $X_{\eta_n}$ est supposée rationnellement connexe.
En particulier, pour $\theta=\eta_1$ le point générique de la base $\mathbb{P}^1,$
les hypothèses sur la fibration $g\circ\pi':X'\to\mathbb{P}^1$
que $Pic(X'_{\bar{\eta}_1})$ est sans torsion, que $Br(X'_{\bar{\eta}_1})$ est fini, et (H CH0)
sont vérifiées, \textit{cf.} la proposition \ref{H-CH0}.
D'après le théorème de Bertini, pour presque tout point fermé
$\theta\in\mathbb{P}^1,$ la fibre $g^{-1}(\theta)=\Delta_\theta$
et $\epsilon^{-1}(\overline{D})$ se rencontrent transversalement, et de plus
la fibre $X'_\theta$ est une variété lisse (on rappelle que $X'$ est lisse).
Pour ces points fermés $\theta,$
on considère $\pi'_\theta:X'_\theta\to\Delta_\theta\simeq\mathbb{P}^{n-1}_{k(\theta)}.$
Toute fibre au-dessus d'un point fermé de
l'ouvert non vide $\Delta_\theta\cap\epsilon^{-1}(U\setminus\{O\})\subset \Delta_\theta$ vérifie
le principe de Hasse (resp. l'approximation faible).
Le lieu de dégénéré $\overline{D}\cap\Delta_\theta$ reste
de codimension au moins $2$ dans $\Delta_\theta\simeq\mathbb{P}_{k(\theta)}^{n-1}.$
D'après l'hypothèse de récurrence, l'obstruction de Brauer-Manin est la
seule pour les zéro-cycles de degré $1$ sur $X'_\theta$ pour ces points fermés $\theta.$
On arrive 'a la conclusion pour le cas où la base est $\mathbb{P}^n$ d'après le théorème \ref{thm10}.

Généralement $X'$ n'est pas une variété lisse.
Notons que $X\to\mathbb{P}^n$ est lisse au-dessus de $U\subset\mathbb{P}^n,$
l'ouvert $X'\times_\Delta\epsilon^{-1}(U)\subset X'$ est alors lisse,
de plus l'ouvert $X'\times_\Delta\epsilon^{-1}(\mathbb{P}^n\setminus O)\subset X'$ est aussi lisse car il est isomorphe
à l'ouvert $X\times_{\mathbb{P}^n}(\mathbb{P}^n\setminus O)\subset X.$
D'après Hironaka, il existe un morphisme
birationnel $\sigma:X''\to X'$ tel que $X''$ soit une variété lisse, et de plus
les différences entre $X''$ et $X'$ se trouvent au-dessus du fermé $O\cap(\mathbb{P}\setminus U),$
ce dernier fermé est de dimension au plus $n-3.$
On fait le même argument
avec $X''$ au lieu de $X'.$ Notons que, en dehors d'un fermé de $\Delta_\theta\simeq\mathbb{P}^{n-1}_{k(\theta)}$
de codimension au moins $(n-1)-(n-3)=2,$ les fibres de $X''_\theta\to\Delta_\theta$
et de $X'_\theta\to\Delta_\theta$ sont les mêmes.
Le théorème \ref{thm10}
s'applique également à la fibration $X''\to\mathbb{P}^1,$ l'argument fonctionne alors.

\emph{Étape 3.}
Montrons ici seulement le cas où la base est $\mathbb{P}^s\times\mathbb{P}^t,$
on peut le généraliser sans difficulté au cas où la base est
$\mathbb{P}^{n_1}\times\ldots\times\mathbb{P}^{n_r}$ par récurrence.

On note $p:\mathbb{P}^s\times_k\mathbb{P}^t\to\mathbb{P}^s$ la projection sur le premier facteur.
On va appliquer le théorème pour la fibration
$p\circ\pi:X\to\mathbb{P}^s$, on vérifie toutes les hypothèses comme suit.
On note $D_s$ le lieu dégénéré de $\mathbb{P}^s.$
Pour tout point schématique $\theta\in\mathbb{P}^s,$ la fibre $X_\theta$ admet un morphisme
$\pi_\theta:X_\theta\to p^{-1}(\theta)=\theta\times_k\mathbb{P}^t.$ Comme $codim(D)\geqslant2,$
si $\theta\in\mathbb{P}^s$ est de codimension 0 ou 1,
le point générique de $p^{-1}(\theta)$ n'est pas contenu dans $D.$ La fibre générique de $\pi_\theta$
est alors géométriquement intègre, d'où,
en regardant leurs corps des fonctions, $X_\theta$ est géométriquement intègre sur $k(\theta).$
Autrement dit, $\theta$ n'appartient pas à $D_s,$ donc $codim(D_s)\geqslant2.$

De plus, soit $\eta_s$ le point générique de la base $\mathbb{P}^s,$ la fibre générique  de
$\pi_{\eta_s}:X_{\eta_s}\to\eta_s\times_k\mathbb{P}^t$
est exactement la fibre générique $X_\eta$ de $\pi,$ alors les conditions sur le groupe de Picard
et sur le groupe de Brauer sont satisfaites. De plus,
une fois que $X_\eta$ est supposée rationnellement connexe,
d'après un résultat de Graber/Harris/Starr \cite[Corollaire 1.3]{GHS}, $X_{\eta_s}$ est une variété
rationnellement connexe sur $k(\eta_s).$

Afin de conclure, il reste à vérifier que

\noindent($\star$) il existe un ouvert non vide
$U_s$ de $\mathbb{P}^s$ tel que l'obstruction de Brauer-Manin soit la seule
pour les zéro-cycles de degré $1$ sur $X_\theta$ pour tout point fermé $\theta\in U_s.$

À la fibration $\pi_\theta:X_\theta\to p^{-1}(\theta)\simeq\mathbb{P}^t_{k(\theta)}(\theta\in U_s),$
on va appliquer encore une fois ce théorème pour obtenir ($\star$). Vérifions les hypothèses sur
$X_\theta\to p^{-1}(\theta)$ comme suit.

Par le théorème de Bertini, il existe un ouvert non vide $U_s$ de $\mathbb{P}^s,$ tel que
pour tout point fermé $\theta\in U_s$ la fibre $X_\theta$ soit une $k(\theta)$-variété lisse, et tel que
$p^{-1}(\theta)$ et $\overline{D}$ se rencontrent transversalement,
d'où $p^{-1}(\theta)\cap\overline{D}$ reste de codimension au moins $2$ dans $p^{-1}(\theta)\simeq\mathbb{P}^t_{k(\theta)}.$

Quitte à restreindre
l'ouvert non vide $U\subset\mathbb{P}^s\times_k\mathbb{P}^t$ mentionné dans l'hypothèse,
on peut supposer que
toute fibre au-dessus d'un point schématique de $U$ est lisse et géométriquement intègre
(et rationnellement connexe si $X_\eta$ est supposée rationnellement connexe), de plus
les conditions sur le groupe de Picard et sur le groupe de Brauer sont satisfaites.
Quitte à restreindre $U_s,$ on peut supposer que pour tout point fermé $\theta\in U_s,$
$U_\theta=U\cap p^{-1}(\theta)$ est un ouvert non vide de $p^{-1}(\theta).$
Donc le point générique de $p^{-1}(\theta)=\theta\times\mathbb{P}^t$ dans $\mathbb{P}^s\times\mathbb{P}^t$
appartient à $U,$ la fibre générique de $\pi_\theta:X_\theta\to p^{-1}(\theta)$ est alors
géométriquement intègre (et rationnellement connexe si $X_\eta$ est supposée rationnellement connexe), par conséquent
$X_\theta$ est aussi géométriquement intègre sur $k(\theta)$ pour tout point fermé $\theta\in U_s.$
De plus, on sait que toute fibre de
$\pi_\theta:X_\theta\to p^{-1}(\theta)$
au-dessus d'un point fermé de $U_\theta$ satisfait alors l'hypothèse arithmétique:
l'obstruction de Brauer-Manin est la seule pour les points rationnels ou pour les zéro-cycles de degré $1.$
Ceci nous permet d'appliquer ce théorème à chaque fibration
$\pi_\theta:X_\theta\to p^{-1}(\theta)\simeq\mathbb{P}^t_{k(\theta)}(\theta\in U_s)$ et obtenir ($\star$),
qui complète la preuve.
\end{proof}

\begin{rem}\label{rem-hilbertien}
Comme indiqué dans la remarque \ref{rem-B}, on peut
également démontrer un énoncé similaire avec un sous-groupe $B\subset Br(X_{\eta})$
(satisfaisant une condition sur la finitude).
Pour ceci, on doit adapter cet argument de la récurrence
avec un contrôle de $B,$ voir \S 3 de l'article de
Harari \cite{Harari3} pour plus de détails.
En particulier, si $Br(X_{\bar{\eta}})$ est fini et si $Pic(X_{\bar{\eta}})$
est sans torsions, en prenant $B=Br(X_\eta)$ on rentre dans le cadre du théorème.
Dans le même article, Harari a adapté
l'argument de la récurrence avec un sous-ensemble hilbertien.
Autrement dit, la même conclusion reste valable si l'on remplace
l'ouvert $U$ dans l'hypothèse par
un sous-ensemble hilbertien généralisé $\textsf{Hil}$ grâce au théorème \ref{thm10}
appliqué à l'étape $n=1.$
\end{rem}

\subsection{Cas où la fibre générique admet un zéro-cycle de degré $1$}
En comparant avec le théorème \ref{Pn-codim2}, si au lieu de supposer la condition sur $codim(D),$
on suppose que la fibre générique contient un zéro-cycle de degré $1,$ on a le résultat similaire suivant.
Dans ce cas-là, seulement l'approximation faible/forte est intéressante, le principe de
Hasse vaut automatiquement pour les zéro-cycles de degré $1.$ Le cas où la base $Y$ est la droite projective
est exactement le théorème \ref{thm15}.

\begin{thm}\label{Pn-section}
Soit $\pi:X\to Y$ une fibration sur un corps de nombres $k$ avec $X$ une variété projective
lisse et géométriquement intègre et $Y$ une variété $k$-rationnelle de dimension au moins $1$
(\textit{i.e.} son corps des fonctions est purement transcendant sur $k$).
On suppose que $Br(X_{\bar{\eta}})$ est fini et $Pic(X_{\bar{\eta}})$ est sans
torsion, où $X_{\eta}$ est la fibre générique de $\pi.$

Supposons que $ind(X_{\eta}/k(\eta))=1.$

On fait l'hypothèse qu'il existe un ouvert non vide $U$ de $Y$ tel que pour tout point fermé
$\theta\in U$ on ait: sur $X_\theta$
l'obstruction de Brauer-Manin est la seule à l'approximation faible
pour les zéro-cycles de degré $1.$

Alors, l'obstruction de Brauer-Manin est la seule à
l'approximation faible
pour les zéro-cycles de degré $1$ sur $X.$

Si de plus $X_\eta$ est rationnellement connexe, on a la même conclusion pour l'approximation forte.
\end{thm}

\begin{proof}
Tout d'abord, on se ramène au cas où  $Y=\mathbb{P}^n$ pour un certain $n\geqslant1.$
En fait, il existe un ouvert non vide $U_1$ (resp. $U_2$)
de $Y$ (resp. de $\mathbb{P}^n$ où $n=dim(Y)\geqslant1$),
et un isomorphisme $U_1\buildrel\simeq\over\to U_2.$ D'après Nagata et Hironaka, il existe une compactification
$\pi':X'\to\mathbb{P}^n$ du morphisme non propre $X\times_{Y}U_1\to U_1\simeq U_2\to\mathbb{P}^n$
telle que $X'$ est une variété régulière.
Les variétés $X'$ et $X$ sont birationnellement équivalentes, les fibres génériques de $\pi$ et de $\pi'$
s'identifient, il donc suffit de démontrer le théorème pour le cas où $Y=\mathbb{P}^n.$

Le cas où $n=1$ est le théorème \ref{thm15}.
La méthode de la récurrence de la preuve du théorème \ref{Pn-codim2}
fonctionne. En fait,
on peut choisir l'ouvert non vide $U$ de $\mathbb{P}^n$ tel que de plus toute fibre au-dessus d'un
point schématique de $U$ soit d'indice $1$ d'après le même argument que le paragraphe 2.1,
la fibre générique de $\pi'_\theta:X'_\theta\to\Delta_\theta$ est donc d'indice $1$ pour tout
point fermé $\theta$ de la base $\mathbb{P}^1.$ On conclut en appliquant le théorème \ref{thm10}(ii)(iii) à la fibration
$g\circ\pi':X'\to\mathbb{P}^1.$
\end{proof}

\begin{rem}\label{section-vs-index}
(i) La même remarque que la remarque \ref{rem-hilbertien} s'applique.

(ii) Harari a montré sans faire la récurrence un énoncé analogue pour les points rationnels avec une base plus générale
qu'une variété $k$-rationnelle, \cite[Théorème 4.3.1]{Harari}. Mais sa méthode ne fonctionne pas pour
les zéro-cycles.
\end{rem}

\subsection{Cas où toute fibre est abélienne-scindée}
Dans le théorème suivant, le cas où $n=1$ est montré par
Colliot-Thélène/Skorobogatov/Swinnerton-Dyer dans \cite[Théorème 4.1]{CT-Sk-SD}.
Le cas général est mentionné par Wittenberg dans la première remarque de la page 135 de \cite{WittenbergLNM},
on confirme cette remarque à l'aide du théorème \ref{Pn-codim2}.

\begin{thm}\label{Pn-codim1}
Soit $\pi:X\to \mathbb{P}^n$ une fibration sur un corps de nombres $k$ avec $X$ une variété projective
lisse et géométriquement intègre.
On suppose que $Br(X_{\bar{\eta}_n})$ est fini et $Pic(X_{\bar{\eta}_n})$ est sans
torsion, où $X_{\eta_n}$ est la fibre générique de $\pi.$

On fait l'hypothèse\ \\
\noindent\emph{\textsc{(Abélienne-Scindée)}} pour tout point $\theta\in\mathbb{P}^n$ de codimension $1,$
il existe une composante irréductible $Y$
de la fibre $X_\theta$ de multiplicité $1$ telle que la fermeture algébrique de $k(\theta)$ dans le corps
de fonctions de $Y$ est une extension abélienne de $k(\theta).$

Supposons qu'il existe un ouvert non vide $U$ de $\mathbb{P}^n$ tel que pour tout point fermé
$\theta\in U,$ la fibre $X_\theta$ satisfait respectivement
\begin{itemize}
\item[(i)]
le principe de Hasse pour les points rationnels ou pour les zéro-cycles de degré $1;$
\item[(ii)]
l'approximation faible pour les points rationnels ou pour les zéro-cycles de degré $1;$
\item[(iii)]
la condition (ii) et de plus $X_{\eta_n}$ est rationnellement connexe.
\end{itemize}

Alors, pour les zéro-cycles de degré $1$ sur $X,$ respectivement l'obstruction de Brauer-Manin est la seule
\begin{itemize}
\item[(i)]
au principe de Hasse;
\item[(ii)]
à l'approximation faible;
\item[(iii)]
à l'approximation forte.
\end{itemize}
\end{thm}

\begin{proof}
On reprend la preuve du théorème 3.4 de Wittenberg \cite{WittenbergLNM}, il suffit de remplacer le théorème 3.25 de
\cite{WittenbergLNM} (\cite[Théorème 1]{Harari3}),
dont on prend $B$ un sous-groupe fini engendrant $Br(X_{\eta_n})$ modulo $Br(k(\eta_n)),$
par sa version pour les zéro-cycle: le théorème \ref{Pn-codim2} ci-dessus.
\end{proof}


\section{Quelques applications}\label{appl}

\noindent\textbf{Fibrés en variétés de Severi-Brauer/en coniques}

Un fibré en variétés de Severi-Brauer est un morphisme dominant $X\to Y$ dont la fibre générique
est une variété de Severi-Brauer définie sur le corps de fonctions $k(Y).$
Un tel fibré satisfait l'hypothèse \emph{\textsc{(Abélienne-Scindée)}}, \textit{cf.} la preuve
du corollaire 3.6 de \cite[pages 117-118]{WittenbergLNM}.

Il résulte du théorème \ref{Pn-codim1} la conséquence suivante.

\begin{cor}\label{application-fibre-en-SB}
L'obstruction de Brauer-Manin est la seule au principe de Hasse et à l'approximation forte
pour les zéro-cycles de degré $1$ sur
toute variété propre lisse et géométriquement intègre birationnellement équivalente à un fibré
en variétés de Severi-Brauer (en particulier, fibré en coniques) au-dessus de l'espace projectif
sur un corps de nombres.
\end{cor}

La première étude du principe de Hasse pour les zéro-cycles de degré $1$
sur un fibré en coniques au-dessus de $\mathbb{P}^1$ est
due à Salberger \cite{Salberger}, il montre que s'il existe des zéro-cycles de degré $1$ localement partout avec
le groupe de Brauer algébrique nul, il existe un zéro-cycle de degré $1$ global.
Le cas d'un fibré en variétés de Severi-Brauer au-dessus de $\mathbb{P}^1$ est montré par
Colliot-Thélène/Swinnerton-Dyer dans \cite{CT-SD}. Il est généralisé ensuite par
Colliot-Thélène/Skorobogatov/Swinnerton-Dyer \cite{CT-Sk-SD}, Frossard \cite{Frossard},
Wittenberg \cite{Wittenberg}.
D'autre part, la proposition analogue pour les points rationnels est établie par Wittenberg dans
\cite{WittenbergLNM}, Corollaire 3.6, en admettant l'hypothèse de Schinzel.

\bigskip
\noindent\textbf{Fibrés en surfaces de Châtelet}

Un fibré en surfaces de Châtelet au-dessus de $\mathbb{P}^n$ est une variété projective lisse
et géométriquement intègre $X$ munie d'un morphisme
dominant $X\to\mathbb{P}^n$ à fibre générique une surface de Châtelet.

\begin{cor}\label{fibre en surfaces de chatelet}
L'obstruction de Brauer-Manin est la seule au principe de Hasse et à l'approximation forte
pour les zéro-cycles de degré $1$ sur tout fibré en surfaces de Châtelet au-dessus de l'espace projectif
sur un corps de nombres.
\end{cor}

\begin{proof}
La fibre générique
$X_\eta$ est définie par une équation (affine) $y^2-a_{t_1,\ldots,t_n}z^2=P_{t_1,\ldots,t_n}(x)$ où
$a_{t_1,\ldots,t_n}\in k(t_1,\ldots,t_n)^*$ et $P_{t_1,\ldots,t_n}(x)\in k(t_1,\ldots,t_n)[x]$ est un polynôme.
La variété $X$ peut-être vue comme un fibré en coniques au-dessus de $\mathbb{P}^{n+1}$ via
les coordonnées $(t_1,\ldots,t_n;x)\in \mathbb{P}^{n}\times\mathbb{P}^1$ à équivalence birationnelle près.
Pour conclure, on applique \ref{Pn-codim1}.
\end{proof}

L'obstruction de Brauer-Manin pour les points rationnels sur certains fibrés en surfaces de Châtelet
au-dessus de $\mathbb{P}^n$ est discutée par Harari dans \cite{Harari2}, Proposition 4.2.1.
Dans \cite{Harari2}, on suppose que le lieu dégénéré $D$ (défini au début de la section \ref{Pn}) est de
codimension au moins $2,$ ici on n'a pas besoin de cette hypothèse.

Concernant les fibrés en surfaces de Châtelet au-dessus d'une courbe de genre positif, Poonen a trouvé
un tel fibré tel que l'obstruction de Brauer-Manin au principe de Hasse
n'est pas la seule pour les points rationnels, \cite{Poonen}.
Cependant, l'existence d'un zéro-cycle global de degré $1$ sur le solide de Poonen
a été montrée par Colliot-Thélène \cite{CTsurPoonen}.

\bigskip
\noindent\textbf{Fibrés en espaces homogènes}

Comme application, on considère la question proposée par Colliot-Thélène
à la fin de \cite{CT95}. Soit $X\To\mathbb{P}^1$ une fibration dont la fibre générique
est une compactification lisse d'un espace homogène d'un groupe algébrique linéaire connexe. On
demande si l'obstruction de Brauer-Manin est la seule au principe de Hasse/à l'approximation faible pour
les zéro-cycles de degré $1$ (resp. points rationnels) sur $X.$

\begin{prop}
Soit $\pi:X\To\mathbb{P}^1$ une fibration dont la fibre générique $X_\eta$
est une compactification lisse d'un espace homogène $Y$ d'un groupe algébrique réductif connexe $G$ sur $k(\mathbb{P}^1)$.
On suppose que toute fibre de $\pi$ contient une composante irréductible de multiplicité un qui est géométriquement intègre.
On fait une des hypothèses suivantes

(1) le stabilisateur d'un point géométrique de $Y$ est connexe,

(2) $G$ est un groupe simplement connexe, et le stabilisateur d'un point géométrique de $Y$ est abélien,

Alors l'obstruction de Brauer-Manin est la seule au principe de Hasse/à l'approximation forte
pour les zéro-cycles de degré $1$
(resp. au principe de Hasse/à l'approximation faible pour les points rationnels).
\end{prop}

\begin{proof}
En notant le résultat principal (Corollaire 2.5) de Borovoi \cite{Borovoi96}, l'assertion résulte du théorème
\ref{thm10} pour les zéro-cycles de degré $1$ (resp. théorème 4.3.1 de Harari \cite{Harari} pour les points rationnels).
En fait, la fibre générique $X_\eta$ est géométriquement unirationnelle, car le groupe réductif $G$
est une variété unirationnelle, elle satisfait alors toutes les hypothèses du théorème
\ref{thm10} (resp. théorème 4.3.1 de \cite{Harari}).
\end{proof}

\begin{rem}
L'hypothèse, que
toute fibre de la fibration contient une composante irréductible de multiplicité un qui est géométriquement intègre,
est assez forte. Si l'on suppose seulement que toute fibre de la fibration contient une composante irréductible
de multiplicité un, le problème devient beaucoup plus difficile, les fibres dégénérées entraînent
de grosses difficultés, \textit{cf.} \cite{CT98} pour connaître l'histoire concernant cette difficulté.
Par ailleurs, si l'on suppose que la fibre générique admet un zéro-cycle de degré $1$ (resp. un $k(\mathbb{P}^1)$-point),
l'existence des fibres dégénérées est permise, on a également une proposition analogue en appliquant le théorème \ref{thm15},
dans ce cas, seule l'approximation faible/forte est intéressante.
\end{rem}

\bigskip
\noindent\textbf{Hypersurfaces cubiques}

La méthode des fibrations appliquée aux problèmes arithmétiques sur les hypersurfaces
a été discutée par Harari dans \S 5.2 de \cite{Harari}.
En remplaçant les théorèmes 4.2.1 et 4.3.1 de \cite{Harari} par les théorèmes \ref{thm10} et \ref{thm15},
on a un analogue pour les zéro-cycles de degré $1$ de presque tous ces résultats.
On laisse le lecteur vérifier les détails: il faut vérifier l'irréductibilité géométrique
de la fibre au point infini $\infty$ de la base $\mathbb{P}^1$ quand on applique le théorème \ref{thm10}.
Par exemple, on a l'énoncé suivant.

\begin{prop}
Soit $X$ une hypersurface cubique lisse de dimension au moins $3.$
Si la conjecture que l'obstruction de Brauer-Manin est la seule au principe de Hasse (resp. à l'approximation faible)
pour les zéro-cycles de degré $1$ sur les surfaces cubiques lisses est vraie,
alors $X$ satisfait le principe de Hasse (resp. l'approximation faible)
pour les zéro-cycles de degré $1.$
\end{prop}

\smallskip
\small
\noindent \textbf{Remerciements.}
Je tiens à remercier D. Harari
pour les nombreuses discussions que nous avons eues pendant la préparation de cet article, et pour son aide sur le français.
La section \ref{Pn} toute entière est inspirée par une conversation avec J.-L. Colliot-Thélène sur les fibrés
en coniques au-dessus du plan projectif, je tiens à le remercier vivement.
Enfin, je remercie J.-L. Colliot-Thélène, O. Wittenberg, B. Poonen
pour leurs commentaires et suggestions.
\normalsize


\bibliographystyle{plain}
\bibliography{mybib1}
\end{document}